\documentclass[review]{elsarticle}

\usepackage{lineno,hyperref}
\modulolinenumbers[5]
\journal{Computers and Mathematics with Applications}

\usepackage{amsmath,amssymb,amsthm,amsfonts,afterpage,float,bm,stmaryrd,paralist,wrapfig,bbm,soul}

\usepackage{cancel}
\usepackage{color}
\usepackage{fancyhdr}
\usepackage{graphics}
\usepackage{hyperref}
\usepackage{graphicx,epsf}
\usepackage{makeidx}
\usepackage{epsfig}
\usepackage{lscape}
\usepackage{subfigure}
\usepackage{listings}
\usepackage{multirow}

\usepackage{empheq,xcolor}
\usepackage{amsmath,eqparbox,xparse,lipsum}

\pssilent

\newtheorem{theorem}{Theorem}[section]
\newtheorem{lemma}{Lemma}[section]

\newtheorem{definition}{Definition}[section]

\numberwithin{equation}{section}
\numberwithin{figure}{section}
\numberwithin{table}{section}

\def\XXint#1#2#3{{\setbox0=\hbox{$#1{#2#3}{\int}$}
\vcenter{\hbox{$#2#3$}}\kern-.51\wd0}}

\begin{document}

\setlength{\pdfpageheight}{\paperheight}
\setlength{\pdfpagewidth}{\paperwidth}
\title{Double Fourier Sphere Methods with Low Rank Approximation for Block Copolymer Systems on Sphere}
\author{Wangbo Luo}
\address{Department of Applied Mathematics, The Hong Kong Polytechnic University, Hung Hom, Kowloon,
Hong Kong}
\author{Yanxiang Zhao}
\address{Department of Mathematics, George Washington University, Washington D.C., 20052}
\fntext[myfootnote]{Corresponding author: yxzhao@email.gwu.edu}

\begin{abstract}

We introduce Double Fourier Sphere (DFS) methods for the Ohta-Kawasaki (OK) and Nakazawa-Ohta (NO) models on a spherical domain, examining their coarsening dynamics and equilibrium pattern formations. We employed DFS for spatial discretization and the second-order Backward Differentiation Formula (BDF2) scheme for time evolution, resulting in an efficient energy-stable scheme to simulate the OK and NO models on the unit sphere. Our numerical experiments reveal various self-assembled patterns, such as single-bubble assemblies in binary systems and double-bubble and mixed-bubble assemblies in ternary systems. These patterns closely resemble experimental biomembrane patterns, demonstrating the effectiveness of the OK model in real-world applications. Additionally, our study explores the relationship between repulsive strength and the number of bubbles in assemblies, confirming the two-thirds law in the OK model. This provides quantitative evidence of how self-assembled patterns depend on system parameters in copolymer systems.
\end{abstract}


\date{\today}
\maketitle

\section{Introduction}\label{Intro}

Block copolymers are macromolecules composed of covalently bonded polymer chains, each containing monomers with different repeating units. These materials show a remarkable ability for self-assembly, forming various nanoscale-ordered structures at thermodynamic equilibrium \cite{Bats_Fredrickson1990,Hamley2004,Botiz_Darling2010}. Such behaviors have demonstrated substantial utility for understanding and predicting pattern morphologies \cite{Bats_Fredrickson1999,Choksi2011}. Over the past few decades, the periodic structures of block copolymer systems have been extensively studied through both experimental and theoretical investigations \cite{Helfand_Wasserman1976,Leibler1980,Meier1969,Ohta_Kawasaki1986}.

More recently, two well-known phase-field models, the Ohta-Kawasaki (OK) model and the Nakazawa-Ohta (NO) model, have attracted attention for their application in studying the dynamics of complex systems, particularly phase separation and pattern formation in block copolymers \cite{LuoZhao_PhysicaD2024,LuoZhao_NumPDE2024,LuoZhao_AAMM2024,Ren_Shoup2020,Wang_Ren_Zhao2019,Wang_Ren2017}. The OK model, introduced in \cite{Ohta_Kawasaki1986}, describes the phase separation in diblock copolymers, also referred to as binary systems. These systems are composed of two distinct subchains, one consisting of species $A$ and the other of species $B$, respectively, with long-range interactions. Building on this, the NO model, proposed in \cite{Nakazawa_Ohta1993}, generalizes the OK framework to describe triblock copolymer systems, or ternary systems, which consist of three subchains, each containing species $A$, $B$, and $C$, respectively. This extension leads to more complex phase separation behaviors.

As a continuation of our study on the numerical methods for the OK and NO models on the unit disk \cite{LuoZhao_AAMM2024}, this work aims to develop numerical methods to explore the coarsening dynamics and pattern formations at the equilibrium state of the OK and NO models on a spherical domain.

{The OK model \cite{Ohta_Kawasaki1986} for the diblock copolymers on the spherical domain (for brevity, hereafter we refer to it as {\it spherical OK model} (SOK)),  is characterized by the following energy functional, which includes a long-range $(-\Delta)^{-1}$-type interaction \cite{Xu_Zhao2019,Xu_Zhao2020}:}
\begin{align}\label{eqn:OK}
        E^{\mathrm{{S}OK}}[u] = \int_{\Omega}\left[\frac{\epsilon}{2}|\nabla_{\mathbb{S}^2} u|^2+\frac{1}{\epsilon}W(u)\right] \text{d}x + \frac{\gamma}{2}\int_{\Omega}|(-\Delta_{\mathbb{S}^2})^{-\frac{1}{2}}(u-\omega)|^2 \ \text{d}x,
\end{align}
with the volume constraint
\begin{align*}
    \int_{\Omega}u\ \text{d}x = \omega|\Omega|,
\end{align*}
where $\Omega = \mathbb{S}^2 \subset \mathbb{R}^3$ is the unit sphere, and $0<\epsilon\ll 1$ is an interface parameter that indicates the system is in deep segregation regime. Here, $u = u(x)$ is a phase field labeling function which represents the density of the species $A$, and the concentration of the species $B$ is implicitly represented by $1-u(x)$ with the assumption of incompressibility for the binary system. The function $W(u) = 18(u^2-u)^2$ is a double well potential which enforces the phase field function $u$ to be equal to 1 inside the interface and 0 outside the interface. In the interfacial region, the function $u$ transitions rapidly but smoothly from 0 to 1. { Operators $\nabla_{\mathbb{S}^2}$ and $\Delta_{\mathbb{S}^2}$  represent the spherical gradient and spherical Laplacian,
\begin{align}
    \nabla_{\mathbb{S}^{2}} f(\phi,\theta)  &= \ \nabla f(\phi,\theta) - \hat{\mathbf{n}} (\hat{\mathbf{n}}\cdot\nabla f(\phi,\theta)), \\
    \Delta_{\mathbb{S}^{2}}f(\phi,\theta) &= (\sin{(\theta)})^{-1} \frac{\partial}{\partial\theta}\left(\sin{(\theta)}\frac{\partial f}{\partial\theta}\right)+(\sin{(\theta)})^{-2} \frac{\partial^2 f}{\partial\phi^2},
\end{align}
with $\theta$ and $\phi$ being the polar and azimuthal angles.} The first integral in (\ref{eqn:OK}) is a local surface energy which represents the short-range interaction between the chain molecules and favors the large domain. In contrast, the second integral in (\ref{eqn:OK}) models the long-range repulsive interaction with $\gamma >0$ being the strength of the repulsive force that favors the smaller size of the species and leads to the microphase separation in the {S}OK model. Finally, $\omega\in(0,1)$ is the relative volume of the species $A$ among the whole domain. Since species $A$ and $B$ are assumed to be incompressible, it suffices to consider $\omega \in (0,1/2)$, as otherwise, one can simply reverse the roles of species $A$ and $B$.

{ Similarly, for the NO model \cite{Nakazawa_Ohta1993} on spherical domain, which we call {\it spherical NO model} (SNO) hereafter, the associated free energy functional with long-range interaction reads:}
\begin{align}\label{eqn:NO}
    E^{\mathrm{{S}NO}}[u_1,u_2] = & \int_{\Omega}\frac{\epsilon}{2}\left(|\nabla_{\mathbb{S}^2} u_1|^2+|\nabla_{\mathbb{S}^2} u_2|^2+\nabla_{\mathbb{S}^2} u_1 \cdot \nabla_{\mathbb{S}^2} u_2\right) \text{d}x + \int_{\Omega} \frac{1}{\epsilon}W_2(u_1,u_2)\ \text{d}x \nonumber\\
    & + \sum_{i,j=1}^{2}\frac{\gamma_{ij}}{2}\int_{\Omega}\left[(-\Delta_{\mathbb{S}^2})^{-\frac{1}{2}}(u_i-\omega_i)\times (-\Delta_{\mathbb{S}^2})^{-\frac{1}{2}}(u_j-\omega_j)\right] \text{d}x,
\end{align}
where $u_i = u_i(x)$, $i=1,2$ are phase field labeling functions which represent the density of species $A$ and $B$, respectively, and the concentration of species $C$ is implicitly described by $1-u_1(x)-u_2(x)$. The potential $W_2(u_1,u_2)$ is defined as 
\begin{align*}
W_2(u_1,u_2): = \frac{1}{2}\Big[W(u_1)+W(u_2)+W(1-u_1-u_2)\Big].
\end{align*}
The parameter $\gamma_{ij} > 0$, for $i,j=1,2$, represents the strength of the long-range repulsive interaction between $i$-th and $j$-th species (i.e., species $A$ and $B$), with the assumption of symmetry $\gamma_{12} = \gamma_{21}$. The volume constraints for $u_1$ and $u_2$ are given by
\begin{align*}
        \int_{\Omega}u_i\ \text{d}x = \omega_i|\Omega|, \ \ i = 1,2.
\end{align*}
All other parameters are similar to those in the {S}OK model (\ref{eqn:OK}). 

To study the coarsening dynamics and pattern formation at the equilibrium state of the SOK and SNO models, we focus on minimizing the energy functionals (\ref{eqn:OK}, \ref{eqn:NO}) with the volume constraints. To implement these constraints to the minimization problems, we introduce penalty terms to the energy functionals, leading to the unconstrained free energy functionals for the penalized SOK and SNO models \cite{Xu_Zhao2019,Choi_Zhao2021},
\begin{align}
    E^{\mathrm{pSOK}}[u]  = & \int_{\Omega}\left[\frac{\epsilon}{2}|\nabla_{\mathbb{S}^2} u|^2+\frac{1}{\epsilon}W(u)\right] \text{d}x + \frac{\gamma}{2}\int_{\Omega}|(-\Delta_{\mathbb{S}^2})^{-\frac{1}{2}}(u-\omega)|^2 \ \text{d}x \nonumber\\
    & + \frac{M}{2}\left(\int_{\Omega}u\ \text{d}x - \omega|\Omega|\right)^2, \label{eqn:pOK} \\
    E^{\mathrm{pSNO}}[u_1,u_2] = & \int_{\Omega}\frac{\epsilon}{2}\left(|\nabla_{\mathbb{S}^2} u_1|^2+|\nabla_{\mathbb{S}^2} u_2|^2+\nabla_{\mathbb{S}^2} u_1 \cdot \nabla_{\mathbb{S}^2} u_2\right) \text{d}x + \int_{\Omega} \frac{1}{\epsilon}W_2(u_1,u_2)\ \text{d}x \nonumber\\
    & + \sum_{i,j=1}^{2}\frac{\gamma_{ij}}{2}\int_{\Omega}\left[(-\Delta_{\mathbb{S}^2})^{-\frac{1}{2}}(u_i-\omega_i)\times (-\Delta_{\mathbb{S}^2})^{-\frac{1}{2}}(u_j-\omega_j)\right] \text{d}x \nonumber\\
    & + \sum_{i=1}^{2} \frac{M_i}{2}\left(\int_{\Omega}u_i\ \text{d}x - \omega_i|\Omega|\right)^2. \label{eqn:pNO}
\end{align}
In this work, by considering the $L^2$ gradient flow for the penalized SOK (\ref{eqn:pOK}) and penalized SNO (\ref{eqn:pNO}), we develop the penalized Allen-Cahn dynamics of the SOK model (pACSOK) for the time evolution $u(x,t)$ with given initial data $u(x,t=0) = u_0(x)$, 
\begin{align}\label{eqn:pACOK}
    \frac{\partial u}{\partial t} =  \epsilon\Delta_{\mathbb{S}^2} u - \frac{1}{\epsilon}W'(u) - \gamma(-\Delta_{\mathbb{S}^2})^{-1}(u - \omega) - M \left(\int_{\Omega}u\ \text{d}x - \omega|\Omega|\right), 
\end{align}
and the penalized Allen-Cahn dynamics of SNO model (pACSNO) for the time evolution $u_i(x,t)$, $i=1,2$ with given initial conditions $u_i(x,t=0) = u_{i,0}(x)$,
\begin{align}\label{eqn:pACNO}
        \frac{\partial u_i}{\partial t}  =  &\  \epsilon\Delta_{\mathbb{S}^2} u_i + \frac{\epsilon}{2}\Delta_{\mathbb{S}^2} u_j- \frac{1}{\epsilon}\frac{\partial W_2}{\partial u_i} - \gamma_{ii}(-\Delta_{\mathbb{S}^2})^{-1}(u_i- \omega_i)\nonumber \\
        & - \gamma_{ij}(-\Delta_{\mathbb{S}^2})^{-1}(u_j - \omega_j)
         - M_i \left(\int_{\Omega}u_i\ \text{d}x - \omega_i|\Omega|\right), \ \ i,j = 1,2. 
\end{align}
Both the pACSOK (\ref{eqn:pACOK}) and pACSNO (\ref{eqn:pACNO}) dynamics satisfy the energy dissipative law
\begin{align}
& \frac{d}{dt}E^{\mathrm{pSOK}}[u] = -\left\| \frac{\partial u}{\partial t}\right\|_{L^2}^{2} \leq 0, \nonumber \\
& \frac{d}{dt}E^{\mathrm{pSNO}}[u_1,u_2]  = - \left\| \frac{\partial u_{1}}{\partial t}\right\|_{L^2}^{2} - \left\| \frac{\partial u_{2}}{\partial t}\right\|_{L^2}^{2}\leq 0, \nonumber
\end{align}
that is, 
\begin{align}
& E^{\mathrm{pSOK}}[u(t_2)] \leq E^{\mathrm{pSOK}}[u(t_1)], \quad \forall t_2\geq t_1\geq0, \nonumber \\
& E^{\mathrm{pSNO}}[u_1(t_2),u_2(t_2)] \leq E^{\mathrm{pSNO}}[u_1(t_1),u_2(t_1)], \quad \forall t_2\geq t_1\geq0. \nonumber
\end{align}

\subsection{Notations}\label{subsec:notation}
In this subsection, we summarize some conventional notations used throughout the paper. For the discussion of energy stability in the rest of the paper, we adopt the quadratic extension $\Tilde{W}(u)$ of $W(u)$, as proposed in \cite{Shen_Yang2010}. This extension ensures that $W''$ possesses a finite upper bound, a critical condition for energy stable schemes in $L^2$ gradient flow dynamics \cite{Shen_Yang2010}. For simplicity in our discussion, we continue to use $W(u)$ to represent $\Tilde{W}(u)$, and define $ L_{W''}:= \|W''\|_{L^{\infty}}$, namely, $L_{W''}$ denoting the upper bound of $|W''|$.

For the spatial domain, we apply the similar notations as in \cite{Xu_Zhao2019}. Let the spherical domain $\Omega = [-\pi,\pi]\times[0,\pi]\subset\mathbb{R}^2$, and the space $\mathring{H}^{s}(\Omega)$ is given by
\begin{align*}
   & \mathring{H}^{s}_{p}(\Omega) = \left\{u(\phi,\theta)\in H^s(\Omega): \ u \text{ is periodic in }\phi, \theta, \text{ and } \int_{\Omega}u\ \text{d}x =0 \right\}, 
\end{align*}
consisting of all functions $u\in H^s(\Omega)$ which satisfy the periodic boundary condition in both $\phi$ and $\theta$ - directions with zero mean. We use $\|\cdot\|_{H^s}$ to represent the standard Sobolev norm and $L^2(\Omega) = H^0(\Omega)$.

The inverse of the spherical Laplacian $(-\Delta_{\mathbb{S}^2})^{-1}: \ \mathring{L}^{2}_{p}(\Omega) \to \mathring{H}^{1}_{p}(\Omega)$ is presented as follows:
\begin{align}\label{eqn:LapInv}
   & (-\Delta_{\mathbb{S}^2})^{-1}v = u. \Longleftrightarrow  -\Delta_{\mathbb{S}^2} u = -u_{\theta\theta}-\frac{\cos{(\theta)}}{\sin{(\theta)}}u_{\theta}-\frac{1}{\sin{(\theta)}^2}u_{\phi\phi} = v.
\end{align}
We denote $\|(-\Delta_{\mathbb{S}^2})^{-1}\|_{L^2}$ as the optimal constant such that $\|(-\Delta_{\mathbb{S}^2})^{-1}f\|_{L^{2}}\leq C\|f\|_{2}$ \cite{Canuto_Springer2006}. Note that the constant $C$ is bounded and depends on $\Omega$ which is a simple consequence of the elliptic regularity for general smooth domains. One can also see Lemma \ref{lemma:discrete_L^2} for the boundedness of the discrete counterpart.

\subsection{Previous Work and Our Contributions}

Since the OK and NO models initially proposed by Ohta, Kawasaki, and Nakazawa \cite{Ohta_Kawasaki1986, Nakazawa_Ohta1993}, extensive theoretical and numerical studies have been focused on the OK and NO models in recent years. Nishiura and Ohnishi \cite{Nishiura_Ohnishi1995} and Ren and Wei \cite{Ren_Wei2000} presented a simpler analogy of a binary inhibitory system and characterized the minimizers of the OK model for diblock copolymers. Choksi \cite{Choksi2011} later conducted an asymptotic analysis of the OK model, establishing the existence of global minimizers and analyzing their properties, and recently Du, Scott, and Xu \cite{du2025ohta} studied the minimizers of a degenerate case of the Ohta–Kawasaki energy.
Additionally, Luo and Zhao \cite{LuoZhao_PhysicaD2024, LuoZhao_NumPDE2024} introduced a nonlocal OK model to investigate novel patterns at equilibrium states. For triblock copolymers, Ren and Wei \cite{Ren_Wei2003_1, Ren_Wei2003_2} derived the ternary system from the NO model and studied a family of local minimizers with lamellar structure. They further investigated the “double bubble” patterns in triblock copolymer systems \cite{RenWei_ARMA2013, RenWei_ARMA2015}. Building on this work, Wang, Ren, and Zhao \cite{Wang_Ren_Zhao2019} explored bubble assemblies in the NO model, examining how system parameters influence patterns both numerically and theoretically. Furthermore, Xu and Du \cite{xu2022ternary,xu2024ohta} studied the one-dimensional global minimizers for the NO model in different parameter regimes. 

Parallel to these theoretical studies, numerical methods were also developed for both OK and NO models. For instance, with the $H^{-1}$ gradient flow method, researchers applied the Cahn-Hilliard dynamics to study the dynamics and equilibrium states for the OK model. Methods such as the implicit midpoint spectral method \cite{Benesova_SINA2014} and the Implicit-Explicit Quadrature (IEQ) method \cite{ChengYangShen_JCP2017} have been applied. More recently, Xu and Zhao \cite{Xu_Zhao2019, Xu_Zhao2020} presented the $L^2$ gradient flow for the energy functional of the OK and NO models,  leading to the corresponding pACOK (\ref{eqn:pACOK}) and pACNO (\ref{eqn:pACNO}) equations to explore the equilibria and dynamics of the binary and ternary systems. Subsequently, Luo, Zhao, and Choi \cite{Choi_Zhao2021,LuoZhao_NumPDE2024} developed the stabilized second-order semi-implicit method combined with Fourier spectral methods to solve the pACOK and pACNO equations in the square domain. These methods are closely related to energy-stable approaches for gradient flow dynamics in various systems \cite{Chen_Conde2012, Hu_Wise2009, Shen_Yang2010, Shen_Wang2012, Wang_Wise2011, Wise_Wang2009, GengLiYeYang_AAMM2022, ChenQianSong_AAMM2023, LiTangZhou_AAMM2024, HouJuQiao_AAMM2024, du2019maximum, du2021maximum}. Recently, Luo and Zhao \cite{LuoZhao_AAMM2024} applied the ultraspherical spectral method \cite{Olver_Townsend2013} to study the OK and NO models in the disk domain, broadening their applications to new geometries.

In recent years, several articles \cite{Ortellado_FrontMater2020, Lee_CMA2022, ma2018wavenumber, cheng2020efficient} have discussed numerical methods for phase-field models on curved surfaces. Jeong and Kim \cite{jeong2015microphase} introduced a finite difference method to study microphase separation in diblock copolymers on curved surfaces using Cahn-Hilliard-type equations. Subsequently, Li, Kim, and Wang \cite{li2017unconditionally} developed a finite element method for solving these equations on various  curved surfaces, and Mohammadi \cite{mohammadi2020meshless} applied the generalized moving least squares (GMLS) technique for phase-field models on spherical surfaces. 

In this work, we focus on the numerical studies of the SOK and SNO models by using spectral method to solve the pACSOK and pACSNO dynamics. To the best of our knowledge, this is the first study to systematically investigate the patterns of the SOK and SNO models, which is essential for understanding the coarsening dynamics and equilibrium states on sphere. Notably, the numerical results for the SOK model effectively resemble patterns observed in multicomponent vesicles \cite{Baumgart_Nature2003}, as shown in Figure (\ref{fig:real_num}). This resemblance demonstrates that the SOK model can effectively predict various patterns at equilibrium for complex interfacial problems in practical applications.

\begin{figure}[htbp]
\begin{center}
\includegraphics[width=.35\textwidth]{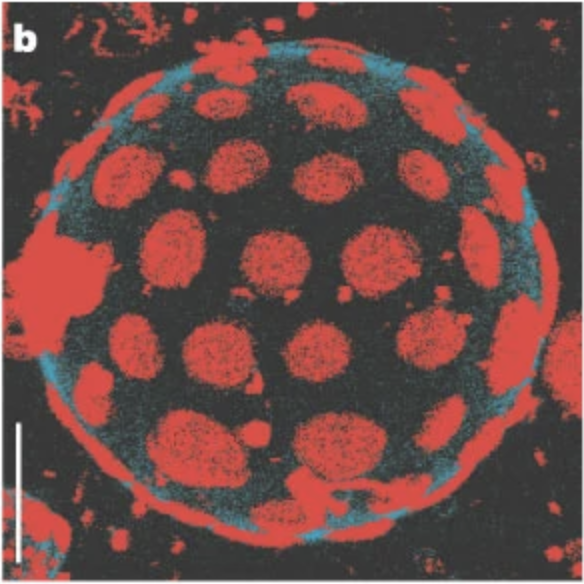} \ \ \ \ \ \ \ \ \ \
\includegraphics[width=.38\textwidth]{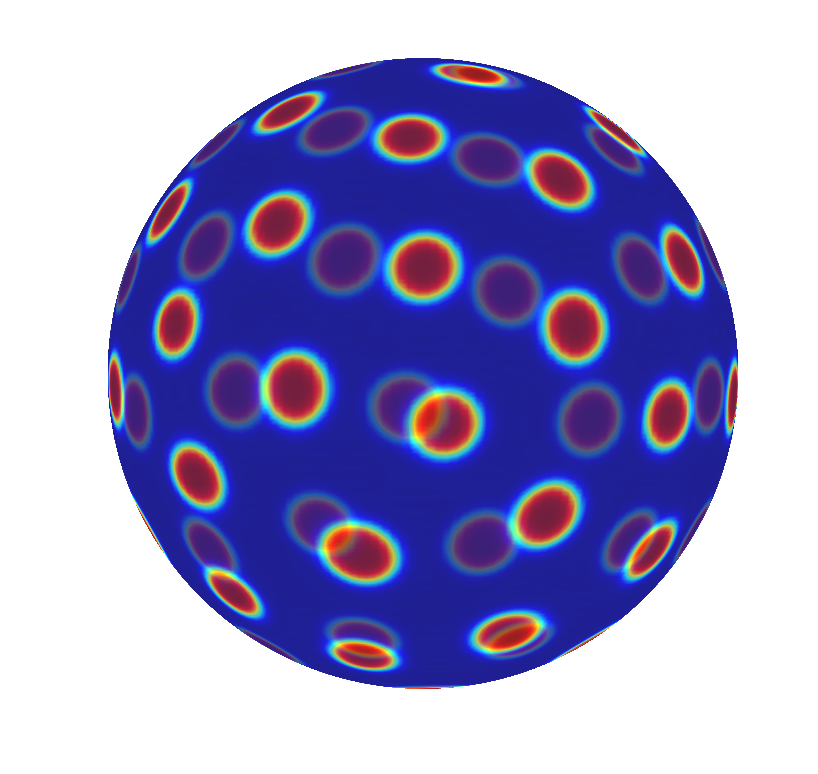}
\end{center}
\caption{Left: Experimental observation of an multicomponent vesicle membrane in \cite{Baumgart_Nature2003}; Right: Numerical simulation for the SOK model.}
\label{fig:real_num}
\end{figure}

The main contribution of our work is three-fold. Firstly, we develop numerical methods for the $L^2$ gradient flow dynamics of the SOK and SNO models to study the binary and ternary systems with long-range interactions on the spherical surface. This approach applies the double Fourier sphere (DFS) method \cite{Merilees_Atm1973} with low rank approximation \cite{TownsendWright_SJSC2013} for spatial discretization and uses the stabilized second-order semi-implicit scheme \cite{Choi_Zhao2021,LuoZhao_NumPDE2024} for time evolution. To our knowledge, this is the first attempt to systematically develop numerical schemes specifically for the SOK and SNO models. Secondly, we analyze the energy stability of the proposed numerical methods, and explore the coarsening dynamics and pattern formation at equilibrium for the SOK and SNO models, offering a solid foundation for further numerical and theoretical studies. Lastly, we systematically explore various self-assembled patterns for the SOK and SNO models.

The rest of the paper is organized as follows. Section \ref{sec:review} provides a review of the DFS method with low-rank approximation and outlines its application to solving the Helmholtz equation in spherical domains. In Section \ref{sec:energy_stability}, we present second-order BDF energy-stable schemes with both semi- and fully-discrete systems for the pACSOK and pACSNO equations and prove the energy stability. Section \ref{sec:numerical} presents numerical experiments for the coarsening dynamics and equilibrium states of the pACSOK and pACSNO systems. Meanwhile, quantitative analysis is conducted to investigate the two-thirds law effect on the long-range interaction strength in the SOK and SNO models. 

%


\section{Preliminary Review}\label{sec:review}
To numerically study the calculus of functions defined on the spherical surface, it is convenient to map these functions onto a rectangular domain using a longitude-latitude coordinate transform. For a function $f(x,y,z)$ in Cartesian coordinates, such a transform is expressed as follows:
\begin{align}
x = \cos{\phi}\sin{\theta}, \ \ \ y = \sin{\phi}\sin{\theta}, \ \ \ z = \cos{\theta}, \ \ \ \ \ (\phi,\theta) \in[-\pi,\pi]\times[0,\pi]. \label{eqn:coordinate_transform}
\end{align}
where $\phi$ is the azimuth angle and $\theta$ is the polar angle. With this transform, the function $f(x,y,z)$ on the unit sphere can be represented by $f(\phi,\theta)$ on the rectangular domain. However, there are two key challenges for this approach. Firstly, oversampling near the poles leads to redundant evaluations. Secondly, the transform does not preserve the periodicity of functions on the sphere in the  $\theta$-direction. 

Recently, Townsend, Wilber and Wright \cite{TownsendWright_SJSC2013} proposed a novel method to address these issues. Their approach combines the double Fourier sphere (DFS) method \cite{Merilees_Atm1973} with low rank approximation, preserving the periodicity of functions on the sphere while avoiding oversampling near the poles. The DFS method, first proposed by Merilees \cite{Merilees_Atm1973} and further developed in \cite{Fornberg1995,Orszag_MWR1974,Yee_MWR1980}, is a powerful and efficient numerical tool. It allows the use of fast Fourier transform (FFT) to map functions defined on the unit sphere onto a rectangular domain, while preserving periodicity in both $\phi$- and $\theta$-directions. This is achieved by doubling the polar angle $\theta$ from $[0, \pi]$ to $[-\pi, \pi]$, thereby transforming $f$ into the so-called Type-I block-mirror-centrosymmetric (BMC-I) function $\tilde{f}$ (refer to Definition 2.1 below). This BMC-I structure ensures that the function $\tilde{f}$ remains smooth and continuous at the poles while preserving its geometric properties.\begin{definition}[BMC-I function \cite{TownsendWright_SJSC2013}]\label{def:BMC-I}
A function $\tilde{f}:[-\pi,\pi]\times[-\pi,\pi]\to\mathbb{C}$ is a Type-I Block-Mirror-Centrosymmetric (BMC-I) function if there are functions $g,h:[0,\pi]\times[0,\pi]\to\mathbb{C}$ such that 
\begin{align}\label{BMC-structure}
\tilde{f} = \begin{bmatrix}
g & h \\
\mathrm{flip}(h) & \mathrm{flip}(g)
\end{bmatrix}
\end{align}
where $\mathrm{filp}$ is a glide reflection in group theory, and $\tilde{f}(\cdot,0) = \alpha$, $\tilde{f}(\cdot,\pi) = \beta$ and $\tilde{f}(\cdot,-\pi) = \gamma$, where $\alpha$, $\beta$ and $\gamma$ are constants. Note that (\ref{BMC-structure}) is called a BMC structure. 
\end{definition}
Next, they \cite{TownsendWright_SJSC2013} introduced the low-rank approximations of BMC-I functions that preserve the BMC-I structure by developing a structure-preserving iterative Gaussian elimination procedure (GE). By using GE, the low-rank approximation of $\tilde{f}(\phi,\theta)$ is given by
\begin{align}
\tilde{f}(\phi,\theta) \approx \sum_{j=1}^{K} a_jc_j(\phi)d_j(\theta), \nonumber
\end{align}
where $a_j$ is a coefficient related to the GE pivots, $c_j(\phi)$ and $d_j(\theta)$ are the $j$th column slice and row slice, respectively, constructed during the GE procedure, and $K$ is a relatively small rank of the approximation \cite{TownsendTrefethen_SJSC2013}. This low-rank approximation samples functions on the spherical surface at a sparse collection of slices that avoids oversampling near the poles, and is particularly effective for numerical operations such as point-wise evaluation, integration, and differentiation, simplifying them to essentially one-dimensional procedures. 

Using this idea, they \cite{TownsendTrefethen_SJSC2013} have built the computational frameworks for differentiation, integration, function evaluation, and vector calculus in spherical geometry, which are publicly available through the open source Chebfun package \cite{driscoll2014chebfun} written in MATLAB. In our work, we directly use this package to work with functions in the spherical domain. This includes evaluating spherical functions on mesh grids, converting spherical functions from and to double Fourier series, and computing integrals of spherical functions. 

\subsection{Fast Solver for Helmholtz Equation in Spherical Domain}\label{subsec:HelmEq_sphere}
Followed by our previous works in the square domain \cite{Choi_Zhao2021,LuoZhao_NumPDE2024,Xu_Zhao2019} and the disk domain \cite{LuoZhao_AAMM2024}, one of the key steps for solving the pACSOK (\ref{eqn:pACOK}) and pACSNO (\ref{eqn:pACNO}) systems is to find $(-\Delta_{\mathbb{S}^2})^{-1}u$ (\ref{eqn:LapInv}) before solving the fully discrete schemes in each time step. To achieve this, an optimal solver for the Helmholtz equation on the sphere can be developed by using the DFS method, detailed in \cite{Cheong_JCP2000,Shen_SISC1999,Wilber_thesis2016}.

Consider the Helmholtz equation 
$$-\Delta_{\mathbb{S}^2} u + \alpha u  = f$$
in spherical coordinates, where $\alpha\ge0$ is a constant. First, we need the zero-mean condition on $f$ to specify a unique solution, i.e. $\int_{0}^{\pi}\int_{-\pi}^{\pi}f(\phi,\theta)\sin{(\theta)}\ \text{d}\phi\text{d}\theta = 0$. Then, we extend the spherical domain from $\Omega = [-\pi,\pi] \times [0,\pi]$ to $\Tilde{\Omega} = [-\pi,\pi] \times [-\pi,\pi]$, resulting in the extended Helmholtz equation  $-\Delta_{\mathbb{S}^2} \tilde{u} + \alpha \tilde{u}  = \tilde{f}$. Here, $\tilde{f}$ is the BMC-I extension of $f$ from (\ref{def:BMC-I}), and the equation can be represented by     
\begin{equation}\label{eqn:2Dhelmtz}
\begin{cases}
  -\sin^2{(\theta)} \frac{\partial^2 \tilde{u}}{\partial \theta^2} - \sin{(\theta)}\cos{(\theta)}\frac{\partial \tilde{u}}{\partial \theta} - \frac{\partial^2\tilde{u}}{\partial \phi^2} + \alpha\sin^2{(\theta)}\tilde{u} = \sin^2{(\theta)}\tilde{f}, & (\phi,\theta)\in \Tilde{\Omega} = [-\pi,\pi]^2, \\
  \tilde{u} \text{ is periodic in }\phi, \theta, & 
\end{cases}
\end{equation}
where the multiplication by $\sin^2{(\theta)}$ aims to avoid the singularity at $\theta=0, \pm\pi$. The actual solution $u$ can be given by restricting $\tilde{u}$ back to $\Omega = [-\pi,\pi] \times [0,\pi]$. Additionally, function $\tilde{u}$ is also of zero mean:
\begin{align}
\int_{0}^{\pi}\int_{-\pi}^{\pi}\tilde{u}(\phi,\theta)\sin{(\theta)}\ \text{d}\phi\text{d}\theta = 0. \label{eqn:zero_mean}  
\end{align}
By using the DFS method, the solution $\tilde{u}$ in (\ref{eqn:2Dhelmtz}) can be represented by 
\begin{align}\label{eqn:CF_exp_BMCII}
    \tilde{u}(\phi,\theta) \approx \sum_{l=-\frac{N_{\phi}}{2}}^{\frac{N_{\phi}}{2}-1}\sum_{k=-\frac{N_{\theta}}{2}}^{\frac{N_{\theta}}{2}-1}\hat{u}_{kl}e^{\text{i}k\theta}e^{\text{i}l\phi},
\end{align}
where $\hat{u}_{kl}$ with $k =  -\frac{N_{\theta}}{2}:\frac{N_{\theta}}{2}-1$, $l =  -\frac{N_{\phi}}{2}:\frac{N_{\phi}}{2}-1$ are the double Fourier coefficients, and $N_{\theta}$ and $N_{\phi}$ are both positive even integers defining the number of uniform mesh points. In order to find the coefficient matrix $\hat{\mathbf{u}} = \{\hat{u}_{kl}\}$, we write the function $\tilde{u}$ and $\tilde{f}$ as follows:
\begin{align}
 \tilde{u}(\phi,\theta) \approx \sum_{l=-\frac{N_{\phi}}{2}}^{\frac{N_{\phi}}{2}-1}\hat{u}_{l}(\theta)e^{\text{i}l\phi}, \ \ \ \tilde{f}(\phi,\theta) \approx \sum_{l=-\frac{N_{\phi}}{2}}^{\frac{N_{\phi}}{2}-1}\hat{f}_{l}(\theta)e^{\text{i}l\phi}, \nonumber
\end{align}
where the functions $\hat{u}_{l}(\theta)$ and $\hat{f}_{l}(\theta)$ are given by
\begin{align}
\hat{u}_{l}(\theta) = \sum_{k=-\frac{N_{\theta}}{2}}^{\frac{N_{\theta}}{2}-1}\hat{u}_{kl}e^{\text{i}k\theta}, \ \ \ \hat{f}_{l}(\theta) = \sum_{k=-\frac{N_{\theta}}{2}}^{\frac{N_{\theta}}{2}-1}\hat{f}_{kl}e^{\text{i}k\theta}, \ \ \ l =  -\frac{N_{\phi}}{2}:\frac{N_{\phi}}{2}-1. \nonumber
\end{align}
With these functions, the equation (\ref{eqn:2Dhelmtz}) separates to a system for each index $l = -\frac{N_{\phi}}{2}:\frac{N_{\phi}}{2}-1$ as
\begin{align}
-\sin^2{(\theta)}\frac{\partial^2 \hat{u}_{l}(\theta)}{\partial \theta^2} - \sin{(\theta)}\cos{(\theta)} \frac{\partial \hat{u}_{l}(\theta)}{\partial \theta} + (l^2+\alpha\sin^2{(\theta)})\hat{u}_{l}(\theta) = \sin^2{(\theta)}\hat{f}_{l}(\theta), \nonumber 
\end{align}
that yields a linear system of $N_{\phi}$ equations as follows
\begin{align}
\Big[-\mathcal{M}[\sin{(\theta)}]^2\mathcal{D}^2 - \mathcal{M}[\sin{(\theta)}]\mathcal{M}[\cos{(\theta)}]\mathcal{D} + (l^2+\alpha\mathcal{M}[\sin{(\theta)}])\Big]\hat{\mathbf{u}}_{l} = \mathcal{M}[\sin{(\theta)}]\hat{\mathbf{f}}_{l},  \label{eqn:linear_sys}
\end{align}
where $\hat{\mathbf{u}}_{l} = [\hat{u}_{kl}]^\mathrm{T}$ and $\hat{\mathbf{f}}_{l} = [\hat{f}_{kl}]^\mathrm{T}$ are the vectors of coefficients in double Fourier basis with $k = -\frac{N_{\theta}}{2}:\frac{N_{\theta}}{2}-1$. The multiplication matrices $\mathcal{D}$, $\mathcal{M}[\sin{(\theta)}]$ and $\mathcal{M}[\cos{(\theta)}]$ are given by 
\begin{align}
& \mathcal{M}[\sin{(\theta)}]  = \mathrm{tridiag}\left(-\frac{\text{i}}{2},0,\frac{\text{i}}{2}\right), \nonumber \\
&\mathcal{M}[\cos{(\theta)}]  = \mathrm{tridiag}\left(\frac{1}{2},0,\frac{1}{2}\right), \nonumber\\
& \mathcal{D} = \mathrm{diag}\left(-\frac{\text{i}N_{\theta}}{2}, \dots, -\text{i}, 0, \text{i}, \dots,\frac{\text{i}(N_{\theta}-2)}{2} \right). \nonumber
\end{align}
The above matrices are developed by the derivative and multiplication operators on the Fourier basis as
\begin{align} 
& \sin{(\theta)} \hat{u}_{l}(\theta) =  \sum_{k=-\frac{N_{\theta}}{2}}^{\frac{N_{\theta}}{2}-1}\frac{-\text{i}\hat{u}_{k-1,l}+\text{i}\hat{u}_{k+1,l}}{2}e^{\text{i}k\theta}, \nonumber\\
& \cos{(\theta)} \hat{u}_{l}(\theta) = \sum_{k=-\frac{N_{\theta}}{2}}^{\frac{N_{\theta}}{2}-1}\frac{\hat{u}_{k-1,l}+\hat{u}_{k+1,l}}{2}e^{\text{i}k\theta},     \nonumber \\
& \frac{\partial \hat{u}_{l}(\theta)}{\partial \theta} = \sum_{k=-\frac{N_{\theta}}{2}}^{\frac{N_{\theta}}{2}-1}\text{i}k\hat{u}_{kl}e^{\text{i}k\theta}. \nonumber
\end{align}
Furthermore, we also need to impose the zero-mean condition (\ref{eqn:zero_mean}) at zero-th mode of the linear system (\ref{eqn:linear_sys}) to ensure the uniqueness of the solution. Since the condition (\ref{eqn:zero_mean}) can be approximated by 
\begin{align}
\int_{0}^{\pi}\int_{-\pi}^{\pi}\tilde{u}(\phi,\theta)\sin{(\theta)}\ \text{d}\phi\text{d}\theta \approx 2\pi\sum_{k=-\frac{N_{\theta}}{2}}^{\frac{N_{\theta}}{2}-1}\hat{u}_{k0}\frac{1+e^{\text{i}k\pi}}{1-k^2}, \quad k\neq \pm 1, \nonumber
\end{align}
we define a vector $\mathbf{v} = [v_{k}]$ with $k = -\frac{N_{\theta}}{2}:\frac{N_{\theta}}{2}-1$, where $v_{\pm 1} = 0$ and $v_{k} = \frac{1+e^{\text{i}k\pi}}{1-k^2}$ for $k\neq\pm1$. Then the zero-mean condition can be imposed by $2\pi\mathbf{v}^\mathrm{T}\hat{\mathbf{u}}_{0} = 0$ to the linear system (\ref{eqn:linear_sys}) at $l=0$ by replacing the zeroth row of the matrix $M =\left[-\mathcal{M}[\sin{(\theta)}]^2\mathcal{D}^2 - \mathcal{M}[\sin{(\theta)}]\mathcal{M}[\cos{(\theta)}]\mathcal{D} + (l^2+\alpha\mathcal{M}[\sin{(\theta)}])\right]$ as $M_{-0}$,
\begin{align}
\begin{bmatrix}
   \mathbf{v}  \\
    M_{-0}
    \end{bmatrix}\hat{\mathbf{u}}_{0} =\begin{bmatrix}
   0  \\
    \hat{\mathbf{f}}_{0,-0}
    \end{bmatrix}. \label{eqn:lin_sys0}
\end{align}
At last, we directly apply the ‘backslash' in MATLAB to solve the $N_{\phi}$ linear systems (\ref{eqn:linear_sys}, \ref{eqn:lin_sys0}), and find the double Fourier coefficients of solution $\tilde{u}$. Detailed application can be found in Section (\ref{subsec:FullyScheme}).


\section{Second Order Time-discrete Energy Stable Scheme}\label{sec:energy_stability}

In this section, in order to systematically explore the coarsening dynamics and the equilibrium states for SOK and SNO models, we propose the semi-discrete (discrete in time) and fully-discrete (discrete in both time and space) schemes for the pACSOK (\ref{eqn:pACOK}) and pACSNO (\ref{eqn:pACNO}) equations and study the energy stability of these proposed schemes. In subsection (\ref{subsec:semi}), we develop the time-discrete scheme by applying the modified second-order BDF method (BDF2) for the temporal discretization.  Further applying the DFS method for the spatial discretization, we obtain the fully-discrete scheme in subsection (\ref{subsec:fully}).

\subsection{Second Order Semi-discrete Scheme for pACSOK and  pACSNO Systems}\label{subsec:semi}

For the time-discrete scheme, we propose a second-order semi-implicit scheme for the pACSOK system (\ref{eqn:pACOK}). Given the initial value  $u^0 = u(0)$ and a uniform step size $\tau$, and adopting the idea from \cite{hou2023linear}, we apply the first-order BDF (BDF1) scheme \cite{Xu_Zhao2019} for the pACSOK system to compute the first-step solution $u^1$ as
\begin{align}\label{eqn:OK_1st_order_scheme}
\frac{u^{1}-u^{0}}{\tau_1}  = \ & \epsilon\Delta_{\mathbb{S}^2} u^{1} - \frac{1}{\epsilon}W'(u^{0})  -\frac{\kappa^{*}}{\epsilon}\left(u^{1}-u^{0}\right) -\gamma\beta^{*}(-\Delta_{\mathbb{S}^2})^{-1}\left(u^{1}-u^{0}\right) \nonumber\\
    & -\gamma(-\Delta_{\mathbb{S}^2})^{-1}\left(u^0-\omega\right) -M\left[\int_{\Omega} u^{0}\ \text{d}x - \omega|\Omega|\right],
 \end{align}
where $\kappa^{*},\beta^{*} \geq0$ are constant stabilizers and the time step size $\tau_1$ for the first iteration is taken as $\tau_1= \min\{\tau^2,1\}$. By subtracting the BDF1 scheme (\ref{eqn:OK_1st_order_scheme}) from the pACSOK system (\ref{eqn:pACOK}) with $u(\tau_1)$ as the exact solution at time $t = \tau_1$ and taking $L^2$ inner product on both sides by $e^{1}:= u(\tau_1)-u^1$, the following error estimate is obtained, as shown in \cite{Xu_Zhao2019}:
\begin{align}
\|u(\tau_1)-u^1\|_{L^2} = O(\tau_1).
\end{align}
Therefore, the choice of $\tau_1= \min\{\tau^2,1\}$ ensures that the error at the first step remains $O(\tau^2)$, thereby preserving the second-order accuracy in time at any later time when adopting BDF2 after the first step.

Next, we introduce the BDF2 method for the pACSOK system (\ref{eqn:pACOK}), using the given uniform step size $\tau$, the initial condition $u^0$ and the precomputed $u^1$. To ensure energy stability, we choose the stabilizer 
\begin{align*}
    u^{n+1}-2u^{n}+u^{n-1}.
\end{align*}
The approximate solution $u^{n+1}$ for $n=1,2,\cdots,N-1$ is then given by:
\begin{align}\label{eqn:pACOK_scheme}
    \frac{3u^{n+1}-4u^{n}+u^{n-1}}{2\tau}  = \ & \epsilon\Delta_{\mathbb{S}^2} u^{n+1} - \frac{1}{\epsilon}\left[2W'(u^n)-W'(u^{n-1})\right] \nonumber\\
    & -\frac{\kappa}{\epsilon}\left(u^{n+1}-2u^{n}+u^{n-1}\right) -\gamma\beta(-\Delta_{\mathbb{S}^2})^{-1}\left(u^{n+1}-2u^{n}+u^{n-1}\right) \nonumber\\
    & -\gamma(-\Delta_{\mathbb{S}^2})^{-1}\left[2u^n-u^{n-1}-\omega\right] -M\left[\int_{\Omega}\left(2u^n-u^{n-1}\right)\ \text{d}x - \omega|\Omega|\right].
\end{align}
Here, $\frac{\kappa}{\epsilon}\left(u^{n+1}-2u^{n}+u^{n-1}\right)$ and $\gamma\beta(-\Delta_{\mathbb{S}^2})^{-1}\left(u^{n+1}-2u^{n}+u^{n-1}\right)$ control the energy stability of the system and $\kappa, \beta\geq0$ are stabilization constants. 

Similarly, for the pACSNO system (\ref{eqn:pACNO}) of the SNO model, we also use the BDF2 scheme with a uniform time step size $\tau$ over time interval $[0,T]$ and an integer $N$. Given the initial data $u_i^{0} =u_{i,0}$ for $i = 1,2$, the first-step solution $u_i^{1}$, $i = 1,2$ is computed by the BDF1 scheme with $\tau_1 = \min\{\tau^2,1\}$ by following a similar approach used in the pACSOK system
\begin{align}\label{eqn:NO_1st_order_scheme}
        \frac{u^{1}_i-u^{0}_i}{\tau_1}  = &\ \epsilon\Delta_{\mathbb{S}^2} u^{1}_i + \frac{\epsilon}{2}\Delta_{\mathbb{S}^2} u^{0+\text{mod}(i+1,2)}_{j} - \frac{1}{\epsilon}\frac{\partial W_2}{\partial u_{i}}(u^{0+\text{mod}(i+1,2)}_1,u^{0}_{2}) \nonumber \\
    &  -\frac{\kappa_{i}^{*}}{\epsilon}\left(u^{1}_i-u^{0}_i\right)  - \gamma_{ii}\beta_{i}^{*}(-\Delta_{\mathbb{S}^{2}})_{}^{-1}\left(u^{1}_i-u^{0}_i\right)  \nonumber\\
    & -\gamma_{ii}(-\Delta_{\mathbb{S}^{2}})_{u}^{-1}\left[u^0_i-\omega_i\right]  -\gamma_{ij}(-\Delta_{\mathbb{S}^{2}})_{u}^{-1}\left[u^{0+\text{mod}(i+1,2)}_j-\omega_j\right]\nonumber\\
    & -M_i\left[\langle u_i^0,1\rangle - \omega_i|\Omega_{}|\right].
\end{align}
Then, the approximate solutions $u^{n+1}_i$, $i=1,2$ for $n = 1,2,\cdots,N-1$ can be found by
\begin{align}\label{eqn:pACNO_scheme}
    &\frac{3u^{n+1}_i-4u^{n}_i+u^{n-1}_i}{2\tau} \nonumber\\ 
    =\ & \epsilon\Delta_{\mathbb{S}^2} u^{n+1}_i + \frac{\epsilon}{2}\left(2\Delta_{\mathbb{S}^2} u^{n+\text{mod}(i+1,2)}_{j}-\Delta_{\mathbb{S}^2} u^{n-1+2\cdot\text{mod}(i+1,2)}_{j}\right) \nonumber \\
    & - \frac{1}{\epsilon}\left[2\frac{\partial W_2}{\partial u_{i}}(u^{n+\text{mod}(i+1,2)}_1,u^{n}_{2})-\frac{\partial W_2}{\partial u_{i}}(u^{n-1+2\cdot\text{mod}(i+1,2)}_1,u^{n-1}_{2})\right] \nonumber \\
    & -\frac{\kappa_i}{\epsilon}\left(u^{n+1}_i-2u^{n}_i+u^{n-1}_i\right) - \gamma_{ii}\beta_{i}(-\Delta_{\mathbb{S}^2})^{-1}\left(u^{n+1}_i-2u^{n}_i+u^{n-1}_i\right)  \nonumber\\
    & -\gamma_{ii}(-\Delta_{\mathbb{S}^2})^{-1}\left[2u^n_i-u^{n-1}_i-\omega_i\right]  -\gamma_{ij}(-\Delta_{\mathbb{S}^2})^{-1}\left[2u^{n+\text{mod}(i+1,2)}_j-u^{n-1+2\cdot\text{mod}(i+1,2)}_j-\omega_j\right]\nonumber\\
    & -M_i\left[\int_{\Omega}\left(2u^n_i-u^{n-1}_i\right)\ \text{d}x - \omega_i|\Omega|\right],
\end{align}
for $i = 1,2$ and $j\neq i$, and $\kappa_1,\kappa_2,\beta_1,\beta_1\geq0$ are stabilization constants with the stabilizer 
\begin{align*}
    u^{n+1}_i-2u^{n}_i+u^{n-1}_i,\ \ i = 1,2.
\end{align*} 
To solve for $u_1^{n+1}$ at $(n+1)$-step, the $\operatorname{mod}$ function is used to extrapolate $u_2$ at $(n-1)$ and $(n)$-steps as $2u_2^{n} - u_2^{n-1}$. In contrast, once $u_1^{n+1}$ is known, the $\operatorname{mod}$ function selects $u_1^{n+1}$ to obtain $u_2^{n+1}$. The $\operatorname{mod}$ function is designed to preserve the second-order convergence rate of the BDF2 scheme, as shown in \cite{Choi_Zhao2021}.

Next, we aim to study the energy stability for the proposed schemes (\ref{eqn:pACOK_scheme}, \ref{eqn:pACNO_scheme}) of the pACSOK and pACSNO equations. We need the Lipschitz constant $L_{W''}$ for the second order derivative of the double well potential $W''$ and the $L^2$ bound of the inverse Laplacian operator $\|(-\Delta_{\mathbb{S}^2})^{-1}\|_{L^2}$ in (\ref{subsec:notation}) to prove the energy stability. 

\begin{theorem}\label{thm:OK/NO_energy} 
For pACSOK equation, we define a modifying energy functional $\tilde{E}^{\mathrm{pSOK}}[u^{n},u^{n-1}]$ for the SOK model
 \begin{align}\label{eqn:OK_modifiedenergy}
\tilde{E}^{\mathrm{pSOK}}(u^{n},u^{n-1} ) = E^{\mathrm{pSOK}}(u^n)+\left(\frac{\kappa}{2\epsilon}+\frac{1}{4\tau}+C\right)\|u^{n}-u^{n-1}\|_{L^2}^2+\frac{\gamma \beta}{2}\|(-\Delta_{\mathbb{S}^2})^{-\frac{1}{2}}\left(u^{n}-u^{n-1}\right)\|_{L^2}^2,
\end{align}
where $\kappa,\beta\geq0$, and the constant $C$ is given by 
\begin{align}\label{cst:pACOK_cont}
    C=\frac{L_{W''}}{2\epsilon}+\frac{\gamma}{2}\|(-\Delta_{\mathbb{S}^2})^{-1}\|_{L^2}+\frac{M}{2}|\Omega|.
\end{align}
Let $\{u^{n}\}_{n=1}^N$ be generated by scheme (\ref{eqn:pACOK_scheme}), and the time step size $\tau\leq\frac{1}{3C}$, then
\begin{align*}
    \tilde{E}^{\mathrm{pSOK}}[u^{n+1},u^{n}]\leq  \tilde{E}^{\mathrm{pSOK}}[u^{n},u^{n-1}].
\end{align*}
For pACSNO systems, we define the modifying energy functional $ \tilde{E}^{\mathrm{pSNO}}[u_1^{n},u_1^{n-1},u_2^{n},u_2^{n-1}]$ as 
\begin{align}\label{eqn:mod_energy_pNO}
     \tilde{E}^{\mathrm{pSNO}}[u_1^{n},&u_1^{n-1},u_2^{n},u_2^{n-1}]=  \nonumber\\
    & E^{\mathrm{pSNO}}[u_1^{n},u_2^{n}]+\sum_{i=1}^{2}\left[\left(\frac{\kappa_i}{2}+\frac{1}{4\tau}+C_i\right)\|u_i^{n}-u_i^{n-1}\|_{L^2}^2+\frac{\gamma_{ii} \beta_{i}}{2}\|(-\Delta_{\mathbb{S}^2})^{-\frac{1}{2}}(u_i^{n}-u_i^{n-1})\|^2_{L^2}\right],
\end{align}
where $\kappa_i,\beta_i\geq0$, $i=1,2$, and the constants $C_i$ are given by 
\begin{align}\label{cst:pNO}
C_i = \frac{L_{W''}}{2\epsilon}+\frac{\gamma_{i1}+\gamma_{i2}}{2} \|(-\Delta_{\mathbb{S}^2})^{-1}\|_{L^2}+\frac{M_i}{2}|\Omega|.
\end{align}
Let $\{u_1^{n},u_2^{n}\}_{n=1}^N$ be generated by scheme (\ref{eqn:pACNO_scheme}), and the time step size $\tau\leq \min\{\frac{1}{3C_1},\frac{1}{3C_2}\}$, then
\begin{align*}
     \tilde{E}^{\mathrm{pSNO}}[u_1^{n+1},u_1^{n},u_2^{n+1},u_2^{n}]\leq  \tilde{E}^{\mathrm{pSNO}}[u_1^{n},u_1^{n-1},u_2^{n},u_2^{n-1}],
\end{align*}
\end{theorem}

\begin{proof}
The proofs of the energy stability for the pACSOK and pACSNO systems are similar. The main difference is that for the pACSNO system, we deal with $i=1,2$ separately and then add them up. Therefore, we omit the proof for the pACSNO system and provide a detailed proof for the pACSOK system. Note that this proof is similar to the one for the disk domain \cite{LuoZhao_AAMM2024}.

Taking $L^2$ inner product with respect to $u^{n+1}-u^n$ on the two sides of (\ref{eqn:pACOK_scheme}), and using the identities $a\cdot(a-b)=\frac{1}{2}a^2-\frac{1}{2}b^2+\frac{1}{2}(a-b)^2$ and $b\cdot(a-b)=\frac{1}{2}a^2-\frac{1}{2}b^2-\frac{1}{2}(a-b)^2$, with $a=u_i^{n+1}-u_i^{n}$, $b=u_i^{n}-u_i^{n-1}$, we obtain two sides as: 
\begin{align}
  \text{ LHS}  =& \ \frac{1}{2\tau}\left\langle2(u^{n+1}-u^n), u^{n+1}-u^n\right\rangle+\frac{1}{2\tau}\left\langle u^{n+1}-2u^{n}+u^{n-1},u^{n+1}-u^n\right\rangle \nonumber\\ 
  = &\ \frac{1}{\tau}\|u_{}^{n+1}-u_{}^{n}\|_{L^2}^2+\frac{1}{4\tau}\left(\|u_{}^{n+1}-u_{}^{n}\|_{L^2}^2-\|u_{}^{n}-u_{}^{n-1}\|_{L^2}^2+\|u_{}^{n+1}-2u_{}^{n}+u_{}^{n-1}\|_{L^2}^2\right) \nonumber \\
  \text{RHS} = & -\frac{\epsilon}{2}\left(\left\langle-\Delta_{\mathbb{S}^2} u_{}^{n+1},u_{}^{n+1}\right\rangle-\left\langle-\Delta_{\mathbb{S}^2} u_{}^{n},u_{}^{n}\right\rangle +\left\langle-\Delta_{\mathbb{S}^2}(u_{}^{n+1}-u_{}^{n}),u_{}^{n+1}-u_{}^{n}\right\rangle  \right) \nonumber\\
    & \underbrace{- \frac{1}{\epsilon}\left\langle 2W'(u^{n}) - W'(n^{n-1}), u_1^{n+1} - u_{}^{n}\right\rangle}_{\text{I}} \nonumber\\
    & -\frac{\kappa_{}}{2\epsilon}\left(\|u_{}^{n+1}-u_{}^n\|_{L^2}^2-\|u_{}^{n}-u_{}^{n-1}\|_{L^2}^2+\|u_{}^{n+1}-2u_{}^n+u_{}^{n-1}\|_{L^2}^2 \right) \nonumber\\
    & -\frac{\gamma_{} \beta_{}}{2}\left(\|(-\Delta_{\mathbb{S}^2})^{-\frac{1}{2}}(u^{n+1}-u^{n})\|_{L^2}^2-\|(-\Delta_{\mathbb{S}^2})^{-\frac{1}{2}}(u^{n}-u^{n-1})\|_{L^2}^2 + \|(-\Delta_{\mathbb{S}^2})^{-\frac{1}{2}}(u^{n+1}-2u^{n}+u^{n-1})\|_{L^2}^2\right) \nonumber\\
    & \underbrace{-\gamma_{}\left\langle(-\Delta_{\mathbb{S}^2})^{-1}\left(2u^{n}-u^{n-1}-\omega\right),u^{n+1}-u^{n}\right\rangle}_{\text{II}}\nonumber\\
    & \underbrace{-M_1\left(\int_{\Omega} (2u^n-u^{n-1})\ \text{d}A - \omega|\Omega|\right)\left\langle 1,u^{n+1}-u^{n}\right\rangle}_{\text{III}},
\end{align}
By Taylor Expansion, 
\begin{align*}
    & W'(u^{n})(u^{n+1}-u^{n})=W(u^{n+1})-W(u^{n})-\frac{W''(\xi^n)}{2}(u^{n+1}-u^{n})^2, \\
   &  W'(u^{n})-W'(u^{n-1})=W''(\mu^n)(u^{n}-u^{n-1}),
\end{align*}
where $\xi^n$ is between $u^{n+1}$ and $u^n$, $\mu^n$ is between $u^{n-1}$ and $u^n$. Then, since $L_{W''} = \|W''\|_{L^{\infty}}$, the term I can be estimated by Young’s inequality as follows:
\begin{align}
\text{I} = & \ -\frac{1}{\epsilon}\left\langle W(u^{n+1}),1\right\rangle + \frac{1}{\epsilon}\left\langle W(u^{n}),1\right\rangle +\frac{W''(\xi^n)}{2\epsilon}\|u^{n+1}-u^n\|_{L^2}^2-\frac{W''(\mu^n)}{\epsilon}\left\langle u^n-u^{n-1},u^{n+1}-u^n\right\rangle \nonumber\\
\leq & \ -\frac{1}{\epsilon}\left\langle W(u^{n+1}),1\right\rangle + \frac{1}{\epsilon}\left\langle W(u^{n}),1\right\rangle  + \frac{L_{W''}}{2\epsilon}\left(\|u^{n}-u^{n-1}\|_{L^2,h}^2+2\|u^{n+1}-u^n\|_{L^2,h}^2\right)\nonumber.
\end{align}
Also, by Young’s inequality, term II and III can be estimated as 
\begin{align}
\text{II} \leq & \ -\frac{\gamma_{}}{2}\left(\left\langle(-\Delta_{\mathbb{S}^2})^{-1}\left(u^{n+1}-\omega\right),u^{n+1}-\omega\right\rangle -\left\langle(-\Delta_{\mathbb{S}^2})^{-1}\left(u^{n}-\omega\right),u^{n}-\omega\right\rangle \right)\nonumber\\
 &\ + \frac{\gamma_{}}{2} \|(-\Delta_{\mathbb{S}^2})^{-1}\|_{L^2} \left(\|u^{n}-u^{n-1}\|_{L^2}^2+2\|u^{n+1}-u^n\|_{L^2}^2\right), \nonumber\\
\text{III} \leq & \ - \frac{M}{2} \left(\left\langle u^{n+1} - \omega,1\right\rangle^2 - \left\langle u^n -\omega ,1\right\rangle^2\right) \nonumber  \\
        &\ +  \frac{M}{2}|\Omega|\left(\|u^{n}-u^{n-1}\|_{L^2}^2+2\|u^{n+1}-u^n\|_{L^2}^2\right)\nonumber.
\end{align}
Note that the penalized OK energy functional (\ref{eqn:pOK}) can be written as 
\begin{align}
E^{\mathrm{pSOK}}(u^{n}) = & \  \frac{\epsilon}{2}\left\langle-\Delta_{\mathbb{S}^2} u^{n},u^{n}\right\rangle+\frac{1}{\epsilon}\left\langle W(u^{n}),1 \right\rangle \nonumber\\
& \ + \frac{\gamma}{2}\left\langle(-\Delta_{\mathbb{S}^2})^{-1}\left(u^{n}-\omega\right),\left(u^{n}-\omega\right)\right\rangle + \frac{M}{2} \left\langle u^{n}-\omega, 1 \right\rangle^2,
\end{align}
Now, reordering terms from the two sides and dropping the unnecessary terms, we can get the following inequality
\begin{align*}
&\left[E^{\mathrm{pSOK}}(u^{n+1})+\left(\frac{\kappa}{2\epsilon}+\frac{1}{4\tau}\right)\|u^{n+1}-u^n\|_{L^2}^2+\frac{\gamma_{} \beta}{2}\left\langle(-\Delta_{\mathbb{S}^2})^{-1}(u^{n+1}-u^{n}),u^{n+1}-u^n\right\rangle\right] \\
& -\left[E^{\mathrm{pSOK}}(u^{n})+\left(\frac{\kappa}{2\epsilon}+\frac{1}{4\tau}\right)\|u^{n}-u^{n-1}\|_{L^2}^2+\frac{\gamma_{} \beta}{2}\left\langle(-\Delta_{\mathbb{S}^2})^{-1}(u^{n}-u^{n-1}),u^{n}-u^{n-1}\right\rangle\right] \\
 \leq &\ \left(\frac{L_{W''}}{2\epsilon}+\frac{\gamma_{}}{2} \|(-\Delta_{\mathbb{S}^2})^{-1}\|_{L^2}+\frac{M}{2}|\Omega|\right)\left(\|u^{n}-u^{n-1}\|_{L^2}^2+2\|u^{n+1}-u^n\|_{L^2}^2\right) -\frac{1}{\tau}\|u^{n+1}-u_{}^{n}\|_{L^2}^2 \\
 = & \  \left(C\left(\|u^{n}-u^{n-1}\|_{L^2}^2+2\|u^{n+1}-u^n\|_{L^2}^2\right) - \frac{1}{\tau}\|u^{n+1}-u_{}^{n}\|_{L^2}^2\right),
\end{align*}
where $C$ are given by
\begin{align}\label{cst:pNO}
C = \frac{L_{W''}}{2\epsilon}+\frac{\gamma}{2} \|(-\Delta_{\mathbb{S}^2})^{-1}\|_{L^2}+\frac{M}{2}|\Omega|,
\end{align}
where $L_{W''}$ and $\|(-\Delta_{\mathbb{S}^2})^{-1}\|_{L^2}$ are constants which can be found in Section (\ref{subsec:notation}). 

Next, adding $ C\left(\|u^{n+1}-u^n\|_{L^2}^2-\|u^{n}-u^{n-1}\|_{L^2}^2\right)$ to both sides of the above inequality, we further have
\begin{align*}
 \tilde{E}^{\mathrm{pSOK}}[u^{n+1},u^{n}] - \tilde{E}^{\mathrm{pSOK}}[u^{n},u^{n-1}] \leq  \left(3C-\frac{1}{\tau}\right)\|u^{n+1}-u^n\|_{L^2}^2,
\end{align*}
where the modified OK energy functional $\tilde{E}^{\mathrm{pOK}}[u^{n},u^{n-1}]$ is defined in (\ref{eqn:OK_modifiedenergy}). Therefore, if $3C - \frac{1}{\tau} \leq 0$, which means $\tau\leq\frac{1}{3C}$, then 
\begin{align*}
    \tilde{E}^{\mathrm{pOK}}[u^{n+1},u^{n}] \leq \tilde{E}^{\mathrm{pOK}}[u^{n},u^{n-1}],
\end{align*}
and the energy stability for the pACSOK system is proved.

\end{proof}

\subsection{Fully-discrete Scheme for pACSOK and pACSNO Dynamics}\label{subsec:fully}
In this subsection, we apply the DFS method for spatial discretization on the unit sphere to construct the fully-discrete schemes for both pACSOK and pACSNO systems.

\subsubsection{Double Fourier Spectral (DFS) Method for Spatial Discretization}\label{subsection:spatial discretization}
For the pACSOK and pACSNO systems on the domain $\Omega = [-\pi,\pi]\times [0,\pi]$, we extend $\Omega$ into $\Tilde{\Omega}= [-\pi,\pi]\times [-\pi,\pi]$ and transform $f$ into BMC-I function $\tilde{f}$. For brevity, we will still use $f$ to denote its BMC-I extension. The double Fourier expansion of $f$ in the spatial domain is expressed as follows: 
\begin{align}\label{eqn:doub_Four_Exp}
  f(\phi,\theta) = \sum_{j,k}\hat{f}_{jk}e^{\text{i}j\theta}e^{\text{i}k\phi},
\end{align}

Then, we use the spectral collocation approximation to discretize the spatial operators on the unit sphere. Choosing two positive even integers $N_{\phi}$ and $N_{\theta}$, and defining the uniform mesh in both $\phi$ and $\theta$ directions with mesh size $h_{\phi} = \frac{2\pi}{N_{\phi}}$ and $h_{\theta} = \frac{2\pi}{N_{\theta}}$, respectively. The set of collocation points is formulated as $\Tilde{\Omega}_{h} = \{(\phi_i,\theta_j); \ i = 1:N_{\phi},j = 1:N_{\theta}\}$ and the index set is given by 
\begin{align}
    & S_{h} = \{(i,j)\in\mathbb{Z}^2;\ i = 1:N_{\phi},j = 1:N_{\theta}\}, \nonumber \\
    & \hat{S}_h = \{(k,l)\in\mathbb{Z}^2;\ k = -\frac{N_{\phi}}{2}:\frac{N_{\phi}}{2}-1,l = -\frac{N_{\theta}}{2}:\frac{N_{\theta}}{2}-1\}. \nonumber
\end{align}
Denote $\mathcal{M}_{h}$ as the collection of grid functions defined on $\Tilde{\Omega}_{h}$:
\begin{align*}
    \mathcal{M}_{h} = \{f:\Tilde{\Omega}_{h}\to \mathbb{R} | \ f_{i+mN_{\phi},j+nN_{\theta}} = f_{i,j}, \forall(i,j)\in S_{h}, \forall (m,n)\in\mathbb{Z}^2\}.
\end{align*}
For any $f,g\in \mathcal{M}_{h}$ and $\mathbf{f} = (f^1,f^2)^\mathrm{T}, \mathbf{g} = (g^1,g^2)^\mathrm{T}\in \mathcal{M}_{h}\times\mathcal{M}_{h}$, we define the discrete $L^2$ inner product $\langle\cdot,\cdot\rangle_{h}$, discrete $L^2$-norm $||\cdot||_{L^2,h}$, and discrete $L^{\infty}$-norm $||\cdot||_{L^{\infty},h}$ as follows:
\begin{align*}
   & \langle f,g \rangle_{h} =h_{\phi} h_{\theta}\sum_{(i,j)\in S_h}f_{ij}g_{ij}, \ \ ||f||_{L^2,h} = \sqrt{\langle f,f \rangle_{h}}, \ \ ||f||_{L^{\infty},h} = \max_{(i,j)\in S_h}|f_{ij}|, \\
   &  \langle \mathbf{f},\mathbf{g} \rangle_{h} = h_{\phi} h_{\theta}\sum_{(i,j)\in S_h}\left(f^1_{ij}g^1_{ij}+f^2_{ij}g^2_{ij}\right), \ \ ||\mathbf{f}||_{L^2,h} = \sqrt{\langle \mathbf{f},\mathbf{f} \rangle_{h}}.
\end{align*}
The discrete double Fourier transform of a function $f\in \mathcal{M}_h$ is defined as
\begin{align*}
    \hat{f}_{kj} = \frac{1}{N_{\phi}N_{\theta}} \sum_{(i,j)\in S_h}f_{ij}e^{-\text{i}k\phi_i}e^{-\text{i}l\theta_j}, \ \ (k,l)\in\hat{S}_h,
\end{align*}
and its inverse transform is given by
\begin{align*}
    f_{ij} = \sum_{(k,l)\in \hat{S}_h} \hat{f}_{kl}e^{\text{i}k\phi_i}e^{\text{i}l\theta_j}, \ \ (i,j)\in S_h.
\end{align*}
Moreover, the spectral approximations to the differential operators $\partial_{\theta\theta}$, $\partial_{\phi}$, $\partial_{\phi\phi}$ follow:
\begin{align*}
   & \{f_{\theta\theta}\}_{ij} = \{\partial_{\theta\theta} f\}_{ij} = \sum_{(k,l)\in \hat{S}_h} -l^2\hat{f}_{kl}e^{\text{i}k\phi_i}e^{\text{i}l\theta_j}, \ \ (i,j)\in S_h, \\
   & \{f_{\phi}\}_{ij} = \{\partial_{\phi} f\}_{ij} = \sum_{(k,l)\in \hat{S}_h} \text{i}k\hat{f}_{kl}e^{\text{i}k\phi_i}e^{\text{i}l\theta_j}, \ \ (i,j)\in S_h, \\
   & \{f_{\phi\phi}\}_{ij} = \{\partial_{\phi\phi} f\}_{ij} = \sum_{(k,l)\in \hat{S}_h} -k^2\hat{f}_{kl}e^{\text{i}k\phi_i}e^{\text{i}l\theta_j}, \ \ (i,j)\in S_h. 
\end{align*}

\subsubsection{Fully-discrete Schemes}\label{subsec:FullyScheme}
In this subsection, we present the fully-discrete schemes for the pACSOK and pACSNO systems by using the BDF2 scheme in time (\ref{subsec:semi}) and the DFS method for spatial discretization.

Denoting the numerical solutions for the pACSOK equation by $U^n \approx u(\phi,\theta;t_n)|_{\Tilde{\Omega}_{h}}$ with time step size $\tau$ and $t_n = \tau_1+(n-1)\tau$ being the time of $n$-th step, the second-order fully-discrete scheme follows: given initial condition $U^{0} = u_0(\phi,\theta)|_{\Tilde{\Omega}_{h}}$, the first-step solution $U^{1}$ is computed by (\ref{eqn:OK_1st_order_scheme}) with DFS discretization in space and with time step size $\tau_1 = \min\{\tau^2,1\}$: 
\begin{align}\label{eqn:BDF_OK_fully_1st}
    \frac{U^{1}-U^{0}}{\tau_1}
    = \ & \epsilon\Delta_{\mathbb{S}^2,h} U^{1} - \frac{1}{\epsilon}W'(U^{0}) -\frac{\kappa_h^{*}}{\epsilon}\left(U^{1}-U^{0}\right) -\gamma\beta_h^{*}(-\Delta_{\mathbb{S}^2,h})^{-1}\left(U^{1}-U^{0}\right) \nonumber\\
    & -\gamma(-\Delta_{\mathbb{S}^2,h})^{-1}\left(U^0-\omega\right)  -M\left[\langle U^0,1\rangle_h - \omega|\Tilde{\Omega}_h|\right],
\end{align}
where $\kappa_h^*, \beta_h^*\geq 0$ are stabilization constants for the BDF1 scheme, and the choice of $\tau_1$ leads to the second-order accuracy of the fully-discrete scheme. Next, find $U^{n+1} = (U^{n+1})_{ij} \in \mathcal{M}_h$ such that   
\begin{align}\label{eqn:BDF_OK_fully}
    \frac{3U^{n+1}-4U^{n}+U^{n-1}}{2\tau}
    = \ & \epsilon\Delta_{\mathbb{S}^2,h} U^{n+1} - \frac{1}{\epsilon}\left[2W'(U^n)-W'(U^{n-1})\right] \nonumber\\
    &  -\frac{\kappa_h}{\epsilon}\left(U^{n+1}-2U^{n}+U^{n-1}\right) -\gamma\beta_h(-\Delta_{\mathbb{S}^2,h})^{-1}\left(U^{n+1}-2U^{n}+U^{n-1}\right) \nonumber\\
    & -\gamma(-\Delta_{\mathbb{S}^2,h})^{-1}\left(2U^n-U^{n-1}-\omega\right)  -M\left[\langle2U^n-U^{n-1},1\rangle_h - \omega|\Tilde{\Omega}_h|\right],
\end{align}
where $\kappa_h,\beta_h\geq0$ are stabilization constants, the stabilizer $\frac{\kappa_h}{\epsilon}\left(U^{n+1}-2U^{n}+U^{n-1}\right)$ controls the growth of $W'$, and the stabilizer $\gamma \beta_h(-\Delta_{\mathbb{S}^2,h})^{-1}\left(U^{n+1}-2U^{n}+U^{n-1}\right)$ dominates the behavior of $(-\Delta_{\mathbb{S}^2,h})^{-1}$. 

Similarly, the second-order fully-discrete scheme for the pACSNO system reads: given initial condition $U_i^{0} = u_{0,i}(\phi,\theta)|_{\Tilde{\Omega}_{h}}$, $i = 1,2$, and the computed first-step solution $U_i^1$ as
\begin{align}\label{eqn:BDF_NO_fully_1st}
        \frac{U^{1}_i-U^{0}_i}{\tau_1}  = &\ \epsilon\Delta_{\mathbb{S}^2,h} U^{1}_i + \frac{\epsilon}{2}\Delta_{\mathbb{S}^2,h} U^{0+\text{mod}(i+1,2)}_{j} - \frac{1}{\epsilon}\frac{\partial W_2}{\partial u_{i}}(U^{0+\text{mod}(i+1,2)}_1,U^{0}_{2}) \nonumber \\
    &  -\frac{\kappa_{i,h}^{*}}{\epsilon}\left(U^{1}_i-U^{0}_i\right)  - \gamma_{ii}\beta_{i,h}^{*}(-\Delta_{\mathbb{S}^2,h})^{-1}\left(U^{1}_i-U^{0}_i\right)  \nonumber\\
    & -\gamma_{ii}(-\Delta_{\mathbb{S}^2,h})^{-1}\left[U^0_i-\omega_i\right]  -\gamma_{ij}(-\Delta_{\mathbb{S}^2,h})^{-1}\left[U^{0+\text{mod}(i+1,2)}_j-\omega_j\right]\nonumber\\
    & -M_i\left[\langle U_i^0,1\rangle_h - \omega_i|\Tilde{\Omega}_{h}|\right].
\end{align}
with the stabilization constants $\kappa^*_{i,h}, \beta^*_{i,h} \geq0$, $i = 1,2$. The numerical solutions $(U_1^{n+1},U_{2}^{n+1}) = \left((U_1^{n+1})_{ij},(U_2^{n+1})_{ij}\right) \in \mathcal{M}_h\times\mathcal{M}_h$ for each $n\in \mathbb{N^+}$ are given by 
\begin{align}\label{eqn:BDF_NO_fully}
        &\frac{3U^{n+1}_i-4U^{n}_i+U^{n-1}_i}{2\tau} \nonumber\\ 
    = & \epsilon\Delta_{\mathbb{S}^2,h} U^{n+1}_i + \frac{\epsilon}{2}\left(2\Delta_{\mathbb{S}^2,h} U^{n+\text{mod}(i+1,2)}_{j}-\Delta_{\mathbb{S}^2,h} U^{n-1+2\cdot\text{mod}(i+1,2)}_{j}\right) \nonumber \\
    & - \frac{1}{\epsilon}\left[2\frac{\partial W_2}{\partial u_{i}}(U^{n+\text{mod}(i+1,2)}_1,U^{n}_{2})-\frac{\partial W_2}{\partial u_{i}}(U^{n-1+2\cdot\text{mod}(i+1,2)}_1,U^{n-1}_{2})\right] \nonumber \\
    &  -\frac{\kappa_{i,h}}{\epsilon}\left(U^{n+1}_i-2U^{n}_i+U^{n-1}_i\right)  - \gamma_{ii}\beta_{i,h}(-\Delta_{\mathbb{S}^2,h})^{-1}\left(U^{n+1}_i-2U^{n}_i+U^{n-1}_i\right)  \nonumber\\
    & -\gamma_{ii}(-\Delta_{\mathbb{S}^2,h})^{-1}\left[2U^n_i-U^{n-1}_i-\omega_i\right]  -\gamma_{ij}(-\Delta_{\mathbb{S}^2,h})^{-1}\left[2U^{n+\text{mod}(i+1,2)}_j-U^{n-1+2\cdot\text{mod}(i+1,2)}_j-\omega_j\right]\nonumber\\
    & -M_i\left[\langle2U_i^n-U_i^{n-1},1\rangle_h - \omega_i|\Tilde{\Omega}_{h}|\right],
\end{align}
for $i = 1,2$ and $j\neq i$, and stabilization constants $\kappa_{1,h},\kappa_{2,h}, \beta_{1,h},\beta_{2,h}\geq0$.

Next, we focus on proving the energy stability of the fully-discrete schemes (\ref{eqn:BDF_OK_fully}, \ref{eqn:BDF_NO_fully}). Similar to the semi-discrete case in Section (\ref{subsec:semi}), the $L^2$-bound for the discrete inverse Laplacian operator $(-\Delta_{\mathbb{S}^2,h})^{-1}$ is required to guarantee the energy stability.  

\subsubsection{Estimation of $\|(-\Delta_{\mathbb{S}^2,h})^{-1}\|_{L^2,h}$}
In this subsection, we prove the $L^2$ boundedness of the discrete inverse spherical Laplacian $(-\Delta_{\mathbb{S}^2,h})^{-1}$, denoted by $\|(-\Delta_{\mathbb{S}^2,h})^{-1}\|_{L^2,h}$ in Lemma (\ref{lemma:discrete_L^2}), which satisfies $\|(-\Delta_{\mathbb{S}^2,h})^{-1}f\|_{L^2,h}\leq C\|f\|_{L^2,h}$, where $f\in\mathcal{M}_h$ and the constant $C$ is a generic constant independent of the mesh size. 

\begin{lemma}\label{lemma:discrete_L^2}
The $L^2$-bound of the discrete inverse spherical Laplacian under double Fourier expansion is given by
\begin{align*}
\|(-\Delta_{\mathbb{S}^{2},h})^{-1}\|_{L^2,h}  \leq C,
\end{align*}
with a generic constant $C$ independent of the mesh size $h_{\phi}$ and $h_{\theta}$.
\end{lemma}

\begin{proof}
Given any functions $u,f \in \mathcal{M}_h$, such that the discrete spherical Laplacian $-\Delta_{\mathbb{S}^{2},h}$ satisfies
\begin{align}
-\Delta_{\mathbb{S}^{2},h} u = f \Longleftrightarrow -\sin^2{(\theta)} \frac{\partial^2 u}{\partial \theta^2} - \sin{(\theta)}\cos{(\theta)}\frac{\partial u}{\partial \theta} - \frac{\partial^2u}{\partial \phi^2} = \sin^2{(\theta)} f,
\end{align}
where $u$ and $f$ are
\begin{align}
u(\theta,\phi) = \sum\limits_{k = -\frac{N_{\phi}}{2}}^{\frac{N_{\phi}}{2}-1} u_{k}(\theta) e^{\text{i}k\phi},\ \ \  f(\theta,\phi) = \sum\limits_{k = -\frac{N_{\phi}}{2}}^{\frac{N_{\phi}}{2}-1} f_{k}(\theta) e^{\text{i}k\phi}, \nonumber
\end{align}
with 
\begin{align}
u_{k}(\theta) = \sum\limits_{j = -\frac{N_{\theta}}{2}}^{\frac{N_{\theta}}{2}-1} \hat{u}_{jk} e^{\text{i}j\theta}, \ \ \ f_{k}(\theta) = \sum\limits_{j = -\frac{N_{\theta}}{2}}^{\frac{N_{\theta}}{2}-1} \hat{f}_{jk} e^{\text{i}j\theta}, \ \ \ \ k = -\frac{-N_{\phi}}{2}: \frac{N_{\phi}}{2}-1, \nonumber
\end{align}
the Poisson equation becomes
\begin{align}
-\Delta_{\mathbb{S}^{2},h} u = \sum\limits_{k = -\frac{N_{\phi}}{2}}^{\frac{N_{\phi}}{2}-1} \left[-\sin^2{(\theta)}\frac{\partial^2 u_k}{\partial \theta^2} - \sin{(\theta)}\cos{(\theta)}\frac{\partial u_k}{\partial \theta} + k^2u_k\right] e^{\text{i}k\phi} = \sum\limits_{k = -\frac{N_{\phi}}{2}}^{\frac{N_{\phi}}{2}-1} \left[\sin^2{(\theta)}f_k\right] e^{\text{i}k\phi} = \sin^2{(\theta)} f. \nonumber
\end{align}
Meanwhile, to specify a unique solution, the function $u$ is required to satisfy $\int_0^{\pi}\int_{-\pi}^{\pi} u(\theta,\phi) \sin(\theta)\text{d}\phi\text{d}\theta = 0$, which implies $u_0$ does not influence the norm of $u$.  
Therefore, for each $k \in \mathcal{I}_{N_{\phi}}:=\{ -\frac{N_{\phi}}{2},\cdots,-1, 1, \cdots, \frac{N_{\phi}}{2}-1\}$, we have 
\begin{align}
-\sin^2{(\theta)}\frac{\partial^2 u_k}{\partial \theta^2} - \sin{(\theta)}\cos{(\theta)}\frac{\partial u_k}{\partial \theta} + k^2u_k = \sin^2{(\theta)}f_k. \nonumber
\end{align}
Define a bilinear form as follows:
\begin{align}
a_h(u_k,u_k) = \left\langle \mathcal{L}u_k, u_k \right\rangle_h = \left\langle \sin^2{(\theta)}f_k, u_k \right\rangle_h. \nonumber
\end{align}
By applying integration by parts under periodic boundary conditions, we have
\begin{align}
\left\langle \sin{(\theta)}\cos{(\theta)}\frac{\partial u_k}{\partial \theta}, u_k \right\rangle_h & = \left\langle \sin{(\theta)}\cos{(\theta)}u_k, \frac{\partial u_k}{\partial \theta} \right\rangle_h  \nonumber \\
& = \left\langle -\frac{\partial }{\partial \theta}\left(\sin{(\theta)}\cos{(\theta)}u_k\right), u_k \right\rangle_h  \nonumber \\
& = \left\langle -\cos{(2\theta)}u_k -\sin{(\theta)}\cos{(\theta)}\frac{\partial u_k}{\partial \theta}, u_k \right\rangle_h,  \nonumber
\end{align}
which leads to
\begin{align}
\left\langle \sin{(\theta)}\cos{(\theta)}\frac{\partial u_k}{\partial \theta}, u_k \right\rangle_h = \left\langle -\frac{1}{2}\cos{(2\theta)}u_k , u_k \right\rangle_h.  \nonumber
\end{align}
Therefore, the bilinear form $\left\langle \mathcal{L}u_k, u_k \right\rangle_h$ can be written as
\begin{align}
\left\langle \mathcal{L}u_k, u_k \right\rangle_h & =  \left\langle -\sin^2{(\theta)}\frac{\partial^2 u_k}{\partial \theta^2} - \sin{(\theta)}\cos{(\theta)}\frac{\partial u_k}{\partial \theta} + k^2u_k, u_k \right\rangle_h  \nonumber \\
& =  \left\langle -\frac{\partial }{\partial\theta}\left(\sin^2{(\theta)}\frac{\partial u_k}{\partial \theta}\right) + 2\sin{(\theta)}\cos{(\theta)}\frac{\partial u_k}{\partial \theta}- \sin{(\theta)}\cos{(\theta)}\frac{\partial u_k}{\partial \theta}+ k^2u_k, u_k \right\rangle_h  \nonumber\\
& =  \left\langle -\frac{\partial }{\partial\theta}\left(\sin^2{(\theta)}\frac{\partial u_k}{\partial \theta}\right) + \sin{(\theta)}\cos{(\theta)}\frac{\partial u_k}{\partial \theta}+ k^2u_k, u_k \right\rangle_h  \nonumber\\
& = \left\langle \sin^2{(\theta)}\left(\frac{\partial u_k}{\partial \theta}\right)^2 + \left(- \frac{1}{2}\cos{(2\theta)}+ k^2\right)u_k^2, 1 \right\rangle_h \nonumber  \\
& \geq \alpha_k \|u_k\|_{L^2,h}^{2}, \ \ \alpha_k = - \frac{1}{2}\cos{(2\theta)}+ k^2 > 0, \ \ k \in\mathcal{I}_{N_{\phi}}. \nonumber
\end{align}
Therefore, by the Cauchy–Schwarz inequality, the above inequality becomes
\begin{align}
\alpha_k \|u_k\|_{L^2,h}^{2} \leq \left\langle \mathcal{L}u_k, u_k \right\rangle_h = \left\langle \sin^2{(\theta)}f_k, u_k \right\rangle_h \leq \|f_k\|_{L^2,h}\|u_k\|_{L^2,h}, \nonumber
\end{align}
which leads to
\begin{align}
\|u_k\|_{L^2,h} \leq \frac{1}{\alpha_k} \|f_k\|_{L^2,h},  \ \ \alpha_k = - \frac{1}{2}\cos{(2\theta)}+ k^2 > 0, \ \ k \in\mathcal{I}_{N_{\phi}}. \nonumber
\end{align}
Consequently, we can derive the $L^2$-bounds of the discrete inverse spherical Laplacian operator $(-\Delta_{\mathbb{S}^{2},h})^{-1}$ as 
\begin{align}
\left\| u \right\|_{L^2,h}^2 = \sum\limits_{k \in\mathcal{I}_{N_{\phi}}}  \|u_{k}\|_{L^2}^2 \leq \sum\limits_{k \in\mathcal{I}_{N_{\phi}}} \frac{1}{\alpha_k^2} \|f_k\|_{L^2}^2 \leq \max\left\{\frac{1}{\alpha_k^2}\right\}_{k \in\mathcal{I}_{N_{\phi}}} \sum\limits_{k \in\mathcal{I}_{N_{\phi}}}  \|f_{k}\|_{L^2}^2 = \max\left\{\frac{1}{\alpha_k^2}\right\}_{k \in\mathcal{I}_{N_{\phi}}} \left\| f \right\|_{L^2}^2, \nonumber 
\end{align}
that results in 
\begin{align}
\left\| u \right\|_{L^2,h} = \|(-\Delta_{\mathbb{S}^{2},h})^{-1} f\|_{L^2} \leq C\left\| f \right\|_{L^2,h},
\end{align}
where $C = \max\left\{\frac{1}{\alpha_k}, \ k=1,2,\cdots, \frac{N_{\phi}}{2}\right\}$ is a generic constant independent of the mesh size $h_{\phi}$ and $h_{\theta}$.
\end{proof}

%


\subsubsection{Energy stability for the fully-discrete schemes}\label{subsec:FullyScheme}
In this subsection, we aim to prove the energy stability for the fully-discrete schemes for the pACSOK (\ref{eqn:BDF_OK_fully}) and pACSNO (\ref{eqn:BDF_NO_fully}) systems. By defining the discrete energy functional for the SOK model 
\begin{align}\label{eqn:pOK_discrete}
E^{\mathrm{pSOK}}_{h}(U^{n}) =  \frac{\epsilon}{2}\|\nabla_{\mathbb{S}^2,h}U^{n}\|_{L^2,h}+\frac{1}{\epsilon}\left\langle W(U^{n}),1 \right\rangle_h + \frac{\gamma}{2}\|(-\Delta_{\mathbb{S}^{2},h})^{-\frac{1}{2}}\left(U^{n}-\omega\right)\|_{L^2,h} + \frac{M}{2} \left(\left\langle U^{n}, 1 \right\rangle_h - \omega|\Omega|\right),
\end{align}
and the discrete energy functional for the SNO model
 \begin{align}\label{eqn:pNO_discrete}
E^{\mathrm{pSOK}}_{h}(U_1^{n},U_2^{n}) =  & \frac{\epsilon}{2}\left(\|\nabla_{\mathbb{S}^{2},h}U_1^{n}\|_{L^2,h}+\|\nabla_{\mathbb{S}^{2},h}U_2^{n}\|_{L^2,h}+\left\langle \nabla_{\mathbb{S}^{2},h}U_1^{n},\nabla_{\mathbb{S}^{2},h}U_2^{n}\right\rangle\right)+\frac{1}{\epsilon}\left\langle W_2(U_1^{n},U_2^{n}),1 \right\rangle_h \nonumber\\
& + \sum\limits_{i,j=1}^{2}\frac{\gamma_{ij}}{2}\left\langle(-\Delta_{\mathbb{S}^2,h})^{-1}\left(U_i^{n}-\omega\right),\left(U_j^{n}-\omega\right)\right\rangle_{h} + \sum\limits_{i=1}^{2}\frac{M_i}{2} \left(\left\langle U_i^{n}, 1 \right\rangle_h - \omega_i|\Omega|\right),
\end{align}
we establish the following energy stability for the fully discrete schemes (\ref{eqn:BDF_OK_fully}, \ref{eqn:BDF_NO_fully}). The proofs of the theorems are similar to that of Theorem \ref{thm:OK/NO_energy}, the only difference is the substitution of $(-\Delta_{\mathbb{S}^{2}})^{-1}$ with $(-\Delta_{\mathbb{S}^{2},h})^{-1}$ and uses the $L^2$-bound of $(-\Delta_{\mathbb{S}^{2},h})^{-1}$ established in Lemma \ref{lemma:discrete_L^2}. We omit the details for brevity.

\begin{theorem}\label{thm:OK/NO_energy_fully} 
For the fully discrete scheme (\ref{eqn:BDF_OK_fully}) of pACSOK system, we define a modified energy functional as 
 \begin{align*}
\tilde{E}_h^{\mathrm{pSOK}}(U^{n},U^{n-1} ) = E_h^{\mathrm{pSOK}}(U^{n})+\left(\frac{\kappa_h}{2\epsilon}+\frac{1}{4\tau}+C_h\right)\|U^{n}-U^{n-1}\|_{L^2,h}^2+\frac{\gamma \beta_h}{2}\|(-\Delta_{\mathbb{S}^{2},h})^{-\frac{1}{2}}\left(U^{n}-U^{n-1}\right)\|_{L^2,h}^2.
\end{align*}
where $\{U^{n}\}_{n=1}^N$ is generated by scheme (\ref{eqn:BDF_OK_fully}). If $\frac{\kappa_h}{2\epsilon}+\frac{1}{4\tau}\ge0$, $\beta_h\ge0$, and $\frac{1}{\tau}\ge 3 C_h$, then
\begin{align*}
    \tilde{E}^{\mathrm{pSOK}}_h (U^{n+1},U^{n} ) \le \tilde{E}^{\mathrm{pSOK}}_h (U^{n},U^{n-1} ),
\end{align*}
where $C_h$ is a generic constant independent of the time step size $\tau$.

For the pACSNO system (\ref{eqn:BDF_NO_fully}), the modified energy functional is defined by
\begin{align*}
\tilde{E}_h^{\mathrm{pSNO}}[U_1^{n},U_1^{n-1},U_2^{n},U_2^{n-1}] =\ E_h^{\mathrm{pSNO}}[U_1^{n},U_2^{n}]+\sum_{i=1}^{2} &\bigg[\left(\frac{\kappa_{i,h}}{2\epsilon}+\frac{1}{4\tau}+C_{i,h}\right)\|U_i^{n}-U_i^{n-1}\|_{L^2,h}^2 \\
&+\frac{\gamma_{ii} \beta_{i,h}}{2}\|(-\Delta_{\mathbb{S}^2,h})^{-\frac{1}{2}}(U_i^{n}-U_i^{n-1})\|^2_{L^2,h}\bigg],
\end{align*}
where $\{U_1^{n},U_2^{n}\}_{n=1}^N$ is generated by scheme (\ref{eqn:BDF_NO_fully}). If $\kappa_{i,h},\beta_{i,h}\geq0$, $i=1,2$, and $\tau\leq \min\{\frac{1}{3C_{1,h}},\frac{1}{3C_{2,h}}\}$, then
\begin{align*}
     \tilde{E}^{\mathrm{pSNO}}[U_1^{n+1},U_1^{n},U_2^{n+1},U_2^{n}]\leq  \tilde{E}^{\mathrm{pSNO}}[U_1^{n},U_1^{n-1},U_2^{n},U_2^{n-1}],
\end{align*}
where $C_{i,h}$, $i = 1,2$ are generic constants independent of the time step size $\tau$.
\end{theorem}

To implement the fully-discrete scheme (\ref{eqn:BDF_OK_fully}), we take $\beta^*_h, \beta_h = 0$, $\kappa^*_h, \kappa_h \gg 1$, the time step size $\tau$, and initial condition $U^0$, then we compute the first-step solution $U^{1}$ from the BDF1 scheme \cite{Xu_Zhao2019} as follows:
\begin{enumerate}
\item Evaluate $(-\Delta_{\mathbb{S}^2,h})^{-1}(U^{0}-\omega)$ using the Helmholtz solver with $\alpha = 0$ in section \ref{subsec:HelmEq_sphere};
\item Insert $(-\Delta_{\mathbb{S}^2,h})^{-1}(U^{0}-\omega)$ into the right hand side of (\ref{eqn:OK_1st_order_scheme}) and solve another Helmholtz equation
    \begin{align}\label{eqn:OK_fully_Scheme}
    \left(-\epsilon\Delta_{\mathbb{S}^2,h}+\frac{1}{\tau_1}+\frac{\kappa^*_h}{\epsilon}\right)U^{1} = F_h^{0},
    \end{align}
where the term $F_h^{0}$ can be explicitly represented by
    \begin{align}\label{eqn:OK_fully_Scheme_RHS}
        F_h^{0} =  \frac{1}{\tau_1}U^0 - \frac{1}{\epsilon}W'(U^0)+\frac{\kappa^*_h}{\epsilon}U^{0} -\gamma(-\Delta_{\mathbb{S}^2,h})^{-1}\left(U^0-\omega\right)  -M\left[\langle U^0,1\rangle_h - \omega|\Tilde{\Omega}_h|\right].
    \end{align}
and the time step size $\tau_1 = \max\{\tau^2,1\}$.
\end{enumerate}
Then we repeat the following algorithm: for $n = 2, 3, \cdots$,
\begin{enumerate}
\item Evaluate $(-\Delta_{\mathbb{S}^2,h})^{-1}(2U^{n}-U^{n-1}-\omega)$ using the Helmholtz solver with $\alpha = 0$ in section \ref{subsec:HelmEq_sphere};
\item Insert $(-\Delta_{\mathbb{S}^2,h})^{-1}(2U^{n}-U^{n-1}-\omega)$ into the right hand side of (\ref{eqn:BDF_OK_fully}) and solve another Helmholtz equation
    \begin{align}\label{eqn:OK_fully_Scheme}
    \left(-\epsilon\Delta_{\mathbb{S}^2,h}+\frac{3}{2\tau}+\frac{\kappa_h}{\epsilon}\right)U^{n+1} = F_h^{n,n-1},
    \end{align}
where the term $F_h^{n,n-1}$ can be explicitly represented by
    \begin{align}\label{eqn:OK_fully_Scheme_RHS}
        F_h^{n,n-1} = & \frac{2}{\tau}U^n-\frac{1}{2\tau}U^{n-1} - \frac{1}{\epsilon}\left[2W'(U^n)-W'(U^{n-1})\right]+\frac{\kappa_h}{\epsilon}\left(2U^{n}-U^{n-1}\right) \nonumber\\
        & -\gamma(-\Delta_{\mathbb{S}^2,h})^{-1}\left(2U^n-U^{n-1}-\omega\right)  -M\left[\langle2U^n-U^{n-1},1\rangle_h - \omega|\Tilde{\Omega}_h|\right].
    \end{align}
\end{enumerate}

Similar algorithm applies to the fully-discrete scheme (\ref{eqn:BDF_NO_fully}) for the pACSNO system with $\beta^*_{i,h}, \beta_{i,h} = 0$ and $\kappa^*_{i,h}, \kappa_{i,h} \gg 1$ for $i = 1,2$.


\section{Numerical Experiments}\label{sec:numerical}
In this section, we present numerical experiments for the pACSOK and pACSNO systems on the domain $\Omega = [-\pi,\pi]\times [0,\pi] \subset \mathbb{R}^2$. We use the DFS method on the spatial domain $\tilde{\Omega} = [-\pi,\pi]\times [-\pi,\pi]$ and the BDF{2} schemes (\ref{eqn:BDF_OK_fully}, \ref{eqn:BDF_NO_fully}) for temporal discretization with proper stabilizers to explore the coarsening dynamics and pattern formation for the SOK and SNO models. For the spherical domain, we take a uniform mesh grid with $N_{\phi} = N_{\theta} = 2^8$ along both azimuth and polar directions, resulting in the mesh size of $h = \frac{2\pi}{N_{\phi}} = \frac{2\pi}{N_{\theta}}$. The stopping criteria for the time iterations are set to 
\begin{align*}
    \frac{\|U^{n+1}-U^{n}\|_{h,L^{\infty}}}{\tau} \leq 10^{-5}
\end{align*}
for pACSOK equation with approximate solution $U^{n}$ at $n$-th step, and 
\begin{align*}
     \frac{\|U_1^{n+1}-U_1^{n}\|_{h,L^{\infty}}}{\tau}+\frac{\|U_2^{n+1}-U_2^{n}\|_{h,L^{\infty}}}{\tau} \leq 10^{-5}
\end{align*}
for pACSNO equation with approximate solution $U_1^{n},U_2^{n}$ at $n$-th step.

In our numerical simulations of the SOK model, we consider a small relative volume $\omega\ll 1$ for one phase (species $A$) and a relatively high repulsive strength $\gamma\gg 1$. Under these conditions, the pACSOK dynamics yields an equilibrium state characterized by the single bubble assembly, where one phase (species $A$) is embedded into the other (species $B$). For the SNO model, with similarly small relative volumes $\omega_1 \ll 1$ for species $A$ and $\omega_2 \ll 1$ for species $B$, we observe several types of bubble assemblies at equilibrium. Here, species $A$ and $B$ are embedded in species $C$, and the equilibria are influenced by different repulsive strengths $(\gamma_{ij})_{i,j=1,2}$  in pACSNO systems. Moreover, we assume equal relative volumes $\omega_{1} = \omega_{2}$ and set the repulsive strengths symmetrically as $\gamma_{11} = \gamma_{22}$, $\gamma_{12} = \gamma_{21}$ in our numerical experiments. In addition, all the numerical experiments are tested using MATLAB on a 2022 MacBook Pro with an M2 chip and 32 GB of memory.

\subsection{Rate of Convergence}\label{sub:Conv_rate}

In this subsection, we validate the rate of convergence of the BDF2 scheme for both pACSOK (\ref{eqn:pACOK}) and pACSNO(\ref{eqn:pACNO}). 

For the pACSOK equation, we initialize a random circle on the unit sphere $\Tilde{\Omega}= [-\pi,\pi]\times [-\pi,\pi]$, with the following initial condition:
\begin{align*}
    U^0(\phi,\theta) = \left\{
    \begin{array}{ll}
       1,  &  \text{if } (\phi-\phi_0)^2 + (\theta-\theta_0)^2 < r_0^2,\\
       0,  &  \text{otherwise},
    \end{array}
    \right.
\end{align*}
where the center $(\phi_0,\theta_0)$ is randomly chosen from $\phi_0 \in [-\pi,\pi]$, $\theta_0\in[0,\pi]$, and the radius is set to $r_0 = \sqrt{2\pi\omega}+0.2$. While using the parameters $\omega = 0.15$, $\gamma = 100$, $\kappa = 2000$, $\beta = 0$ and $M = 1000$, we implement the fully discrete scheme (\ref{eqn:pACOK_scheme}) for the numerical experiments. The convergence test is performed with $\epsilon=10h$ up to time $T=0.01$, with the benchmark solution computed using a time step size of $\tau = 1\times 10^{-6}$. Table \ref{table:pACOK/NO_rateofconv} presents the error and corresponding rate of convergence at a fixed time $T=0.01$ on the left. It can be observed that the numerical rates for $\epsilon = 10h$ are approximately equal to $2$ for small time step size $\tau$, which aligns with theoretical predictions.

Similarly, for the pACSNO system, the initial condition consists of two random circles representing $U_1$ and $U_2$, respectively, on the unit sphere $\Tilde{\Omega}= [-\pi,\pi]\times [-\pi,\pi]$:
\begin{align*}
        U_1^0(\phi,\theta) = \left\{
    \begin{array}{ll}
       1,  &  \text{if } (\phi-\phi_1)^2 + (\theta+\theta_1)^2 < r_1^2,\\
       0,  &  \text{otherwise},
    \end{array}
    \right. \\
        U_2^0(\phi,\theta) = \left\{
    \begin{array}{ll}
       1,  &  \text{if } (\phi+\phi_2)^2 + (\theta-\theta_2)^2 < r_2^2,\\
       0,  &  \text{otherwise},
    \end{array}
    \right.
\end{align*}
where the centers $(\phi_1,\theta_1)$ and $(\phi_2,\theta_2)$ are randomly chosen from $\phi_1, \phi_2 \in [-\pi,\pi]$, $\theta_1, \theta_2\in[0,\pi]$, with radii $r_1 = \sqrt{2\pi\omega_1}+0.1$ and $r_2 = \sqrt{2\pi\omega_2}+0.1$. In the numerical simulation, we set parameters $\omega_{1} = \omega_{2} = 0.09$, $\gamma_{11} = \gamma_{22} = 100$, $\gamma_{12} = \gamma_{21} = 0$, $\kappa_1 = \kappa_2 = 2000$, $\beta_1 = \beta_2 = 0$ and $M_1 = M_2 = 1000$ and apply the fully discrete scheme (\ref{eqn:pACNO_scheme}). The test runs up to $T = 0.01$ for $\epsilon = 10h$, with the benchmark solution calculated using a time step size of $\tau = 1\times 10^{-6}$. The right side of Table \ref{table:pACOK/NO_rateofconv} displays the errors and convergence rates for the pACSNO equations at $T=0.01$. Similar to the pACSOK equation, the numerical rates for a small time step size $\tau$ closely match the theoretical value of $2$.

\begin{table}[h!]
\begin{center}
\begin{tabular}{ |c||c|c|c|c|  }
\hline
 & \multicolumn{2}{|c|}{pACSOK equation} & \multicolumn{2}{|c|}{pACSNO equation} \\
\hline
$\tau$ & $\|U^{\tau}-U^{\text{benchmark}}\|_{L^2,h}$ & Rate & $\sum_{i=1}^{2}\|U_{i}^{\tau}-U_{i}^{\text{benchmark}}\|_{L^2,h}$ & Rate \\
\hline
5.000e-4 & 1.011190e-1 & -       & 2.546882e-1 &  - \\
2.500e-4 & 3.040393e-2 & 1.733724 & 8.044150e-2 & 1.662720\\
1.250e-4 & 8.989498e-3 & 1.757945 & 2.369693e-2 & 1.763240\\
6.250e-5 & 2.648314e-3 & 1.763166 & 7.244459e-3 & 1.709751\\
3.125e-5 & 7.162332e-4 & 1.886573 & 2.094053e-3 & 1.790581\\
1e-6 (benchmark) & - &- &-&- \\
\hline
\end{tabular}
\end{center}
\caption{Rate of convergence for the fully-discrete schemes of the pACSOK (\ref{eqn:BDF_OK_fully}) with parameters $\omega = 0.15$, $\gamma = 100$, $\kappa = 2000$, $\beta = 0$ and $M = 1000$, and the pACSNO (\ref{eqn:BDF_NO_fully}) with parameters $\omega_{1} = \omega_{2} = 0.09$, $\gamma_{11} = \gamma_{22} = 100$, $\gamma_{12} = \gamma_{21} = 0$, $\kappa_{1} = \kappa_{2} = 2000$, $\beta_{1}=\beta_{2} = 0$ and $M_1 = M_2 = 1000$.}
\label{table:pACOK/NO_rateofconv}
\end{table}

\subsection{Energy Stability for SOK and SNO Models}\label{sub:Energy_stab}

In this subsection, we study the energy stability for the pACSOK and pACSNO systems using the schemes (\ref{eqn:BDF_OK_fully}) and (\ref{eqn:BDF_NO_fully}), respectively. 

For the pACSOK equation, the initial data is randomly generated on a uniform mesh of the mapped spherical domain $\tilde{\Omega} = [-\pi,\pi]\times [-\pi,\pi]$ with a mesh size $16h$. The MATLAB command for such random initials is 
\begin{align*}
    \left[\begin{array}{cc}
        \texttt{repelem}(\texttt{rand}(N_{\phi}/\texttt{ratio},N_{\theta}/(\texttt{ratio}/4)),\texttt{ratio},\texttt{ratio})^\mathrm{T}
    \end{array}\right]
\end{align*}
with ratio = $16$ used for all binary system experiments, each starting from a random initial state. 

For the pACSNO equation, phase separation of the ternary system can be computationally expensive when starting from a fully random initial state similar to the binary system, often costing several days for convergence. To avoid this, we construct a 'semi-random' initial condition to reduce the computation time. This approach partitions the spatial domain $\tilde{\Omega} = [-\pi,\pi]\times [-\pi,\pi]$ into a random number of small patches, and we randomly place circles with random centers and radii in each patch, allowing overlaps between circles. 

In the binary system, we set the following parameters fixed: the time step size $\tau = 1\times10^{-3}$, relative area $\omega = 0.15$, stabilization constant $\kappa = 2000$, $\beta = 0$, and penalty constant $M = 1000$. For the ternary system, the corresponding parameters are fixed by $\tau = 2\times10^{-4}$, $\omega_1=\omega_2 = 0.09$, $\kappa_1 = \kappa_2 = 2000$, $\beta_1=\beta_2=0$ and $M_1 =M_2= 1000$.

Figure \ref{fig:binary_system_energy} presents the energy stability of the pACSOK equation with the numerical scheme (\ref{eqn:BDF_OK_fully}) and displays the equilibrium state for the binary system with repulsive strength $\gamma = 1500$. On the top figure of (\ref{fig:binary_system_energy}), the energy curve monotonically decreases towards the equilibrium state. In each subfigure, we insert snapshots taken at the initial state and at four different times $t$, marked by colored titles corresponding to markers on the energy curve. Starting from the random initial data, the phase separation arises rapidly, leading to the formation of bubbles of different sizes (t = 5). Over time, these bubbles grow into the same size on the surface of a unit sphere, eventually reaching equal sizes (t = 15, 40, 70), similar to previous results observed on square and disk domains \cite{Choi_Zhao2021, LuoZhao_AAMM2024}. The bottom figure in Figure \ref{fig:binary_system_energy} displays three views (front-to-back, left-to-right, and top-to-bottom) of the equilibrium state with 18 single bubbles in the binary system. Further numerical results on the equilibrium states for the SOK model are detailed in subsection (\ref{sub:Equi_OKNO}).

Similarly, Figures \ref{fig:ternary_system_energy_6b} and \ref{fig:ternary_system_energy_8b}, with parameters $\gamma_{11}=\gamma_{22} = 350$, $\gamma_{12}= 0$ and $\gamma_{11}=\gamma_{22} = 500$, $\gamma_{12}= 0$, respectively, show the energy stability for the pACSNO equation using scheme (\ref{eqn:BDF_NO_fully}). The insets in these figures display snapshots taken from a semi-random initial state and at four different times $t$, which display the progress of phase separation to the double-bubble patterns ($t=1$), with double bubbles growing simultaneously until reaching maximum separation across the sphere. The bottom figures in Figures (\ref{fig:ternary_system_energy_6b}) and (\ref{fig:ternary_system_energy_8b}) show three views (front-to-back, left-to-right, and top-to-bottom) of the equilibrium state for the ternary system with 6 and 8 bubbles, respectively. Additional numerical results of the equilibrium states for the SNO model are provided in subsection (\ref{sub:Equi_OKNO}).

\begin{figure}[htbp]
\begin{center}
\includegraphics[width=.8\textwidth]{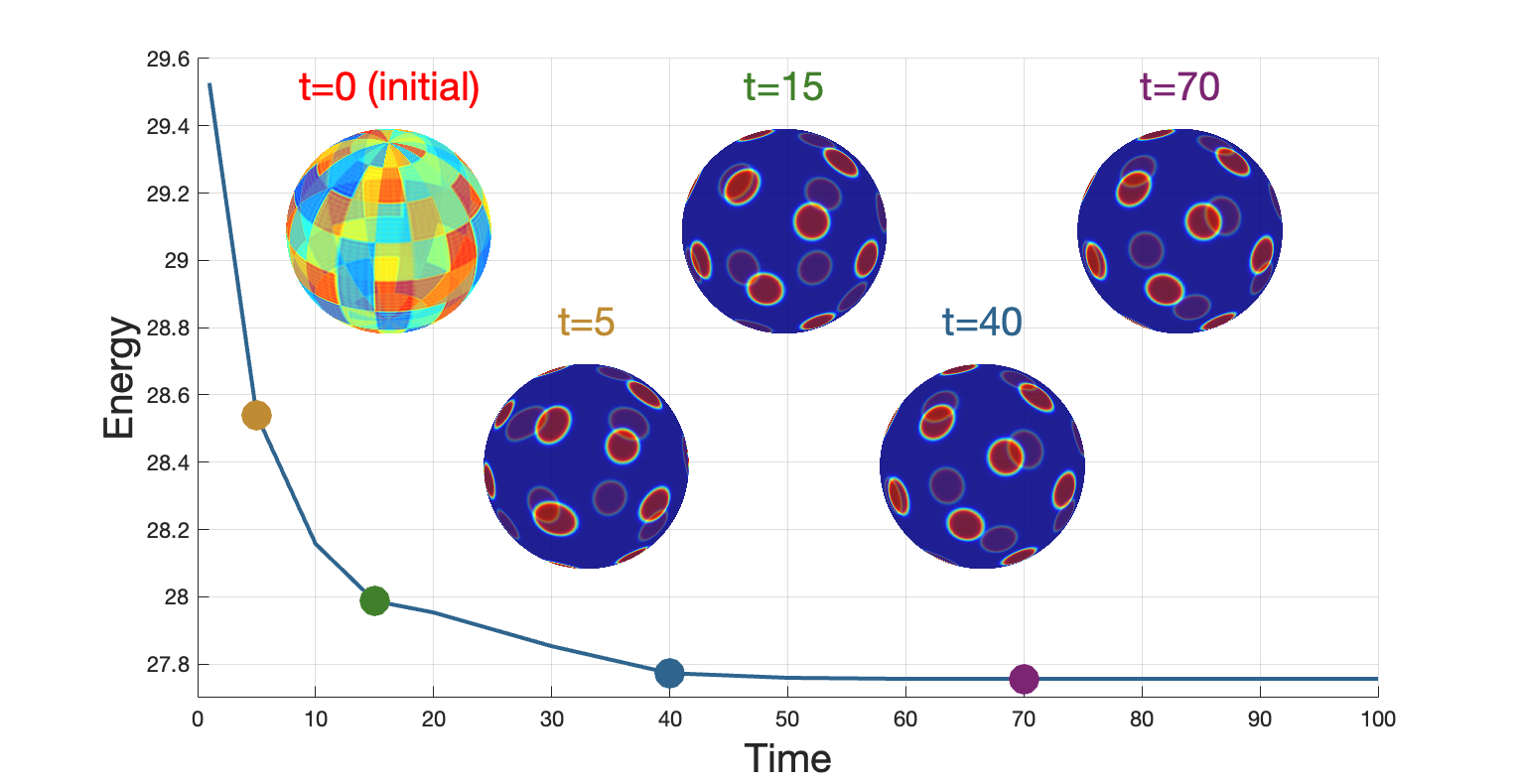}
\includegraphics[width=.8\textwidth]{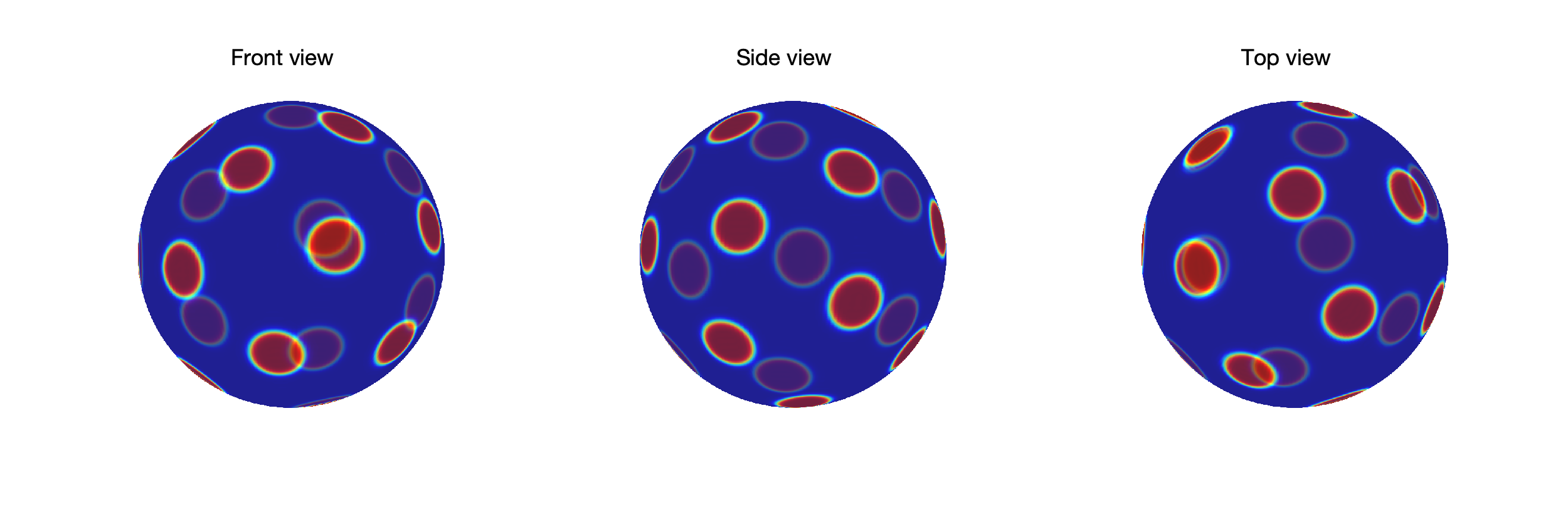}
\end{center}
\caption{(Top) Coarsening dynamics and energy stability for the SOK model with 18 double bubbles, and the corresponding parameters $\omega_{1} = 0.15$, $\gamma=1500$, $\kappa = 2000$, $\beta = 0$ and $M = 1000$. (Bottom) Three views (front view, side view, top view) of the equilibrium state with 18 bubbles.}
\label{fig:binary_system_energy}
\end{figure}

\begin{figure}[htbp]
\begin{center}
\includegraphics[width=.8\textwidth]{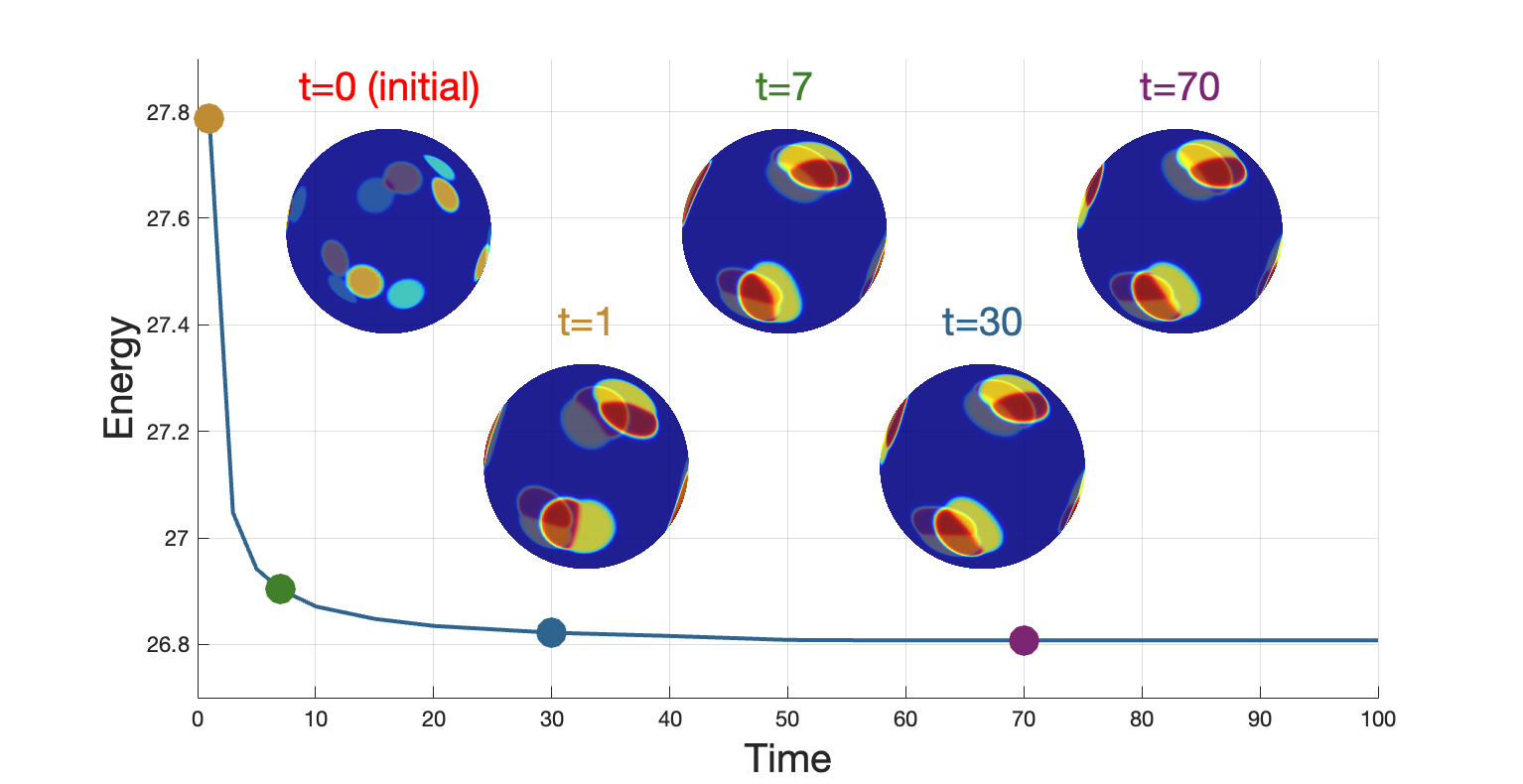}
\includegraphics[width=.8\textwidth]{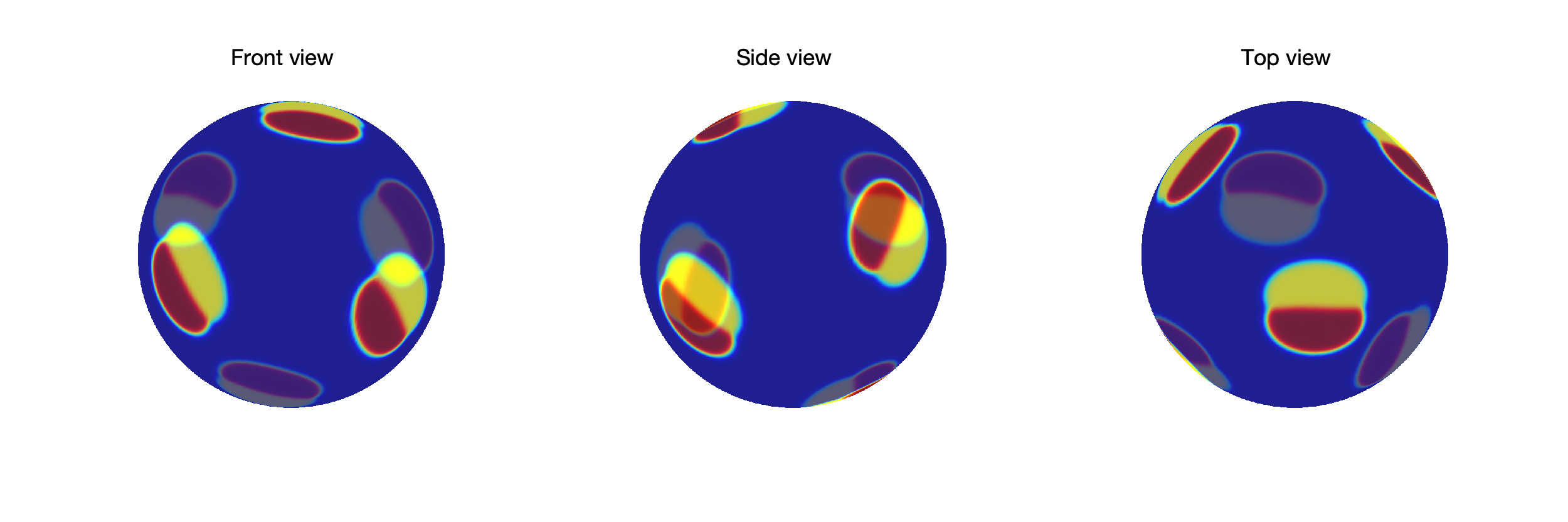}
\end{center}
\caption{(Top) Coarsening dynamics and energy stability for the SNO model with 6 double bubbles. The values of the parameters are $\omega_{1} = \omega_{2} = 0.09$, $\gamma_{11} =  \gamma_{22}=350, \gamma_{12} = \gamma_{21}=0$, $\kappa_1 = \kappa_2 = 2000$, $\beta_1 = \beta_2 = 0$ and $M_1 = M_2 = 1000$. (Bottom) Three views (front view, side view, top view) of the equilibrium state with 6 bubbles.}
\label{fig:ternary_system_energy_6b}
\end{figure}

\begin{figure}[htbp]
\begin{center}
\includegraphics[width=.8\textwidth]{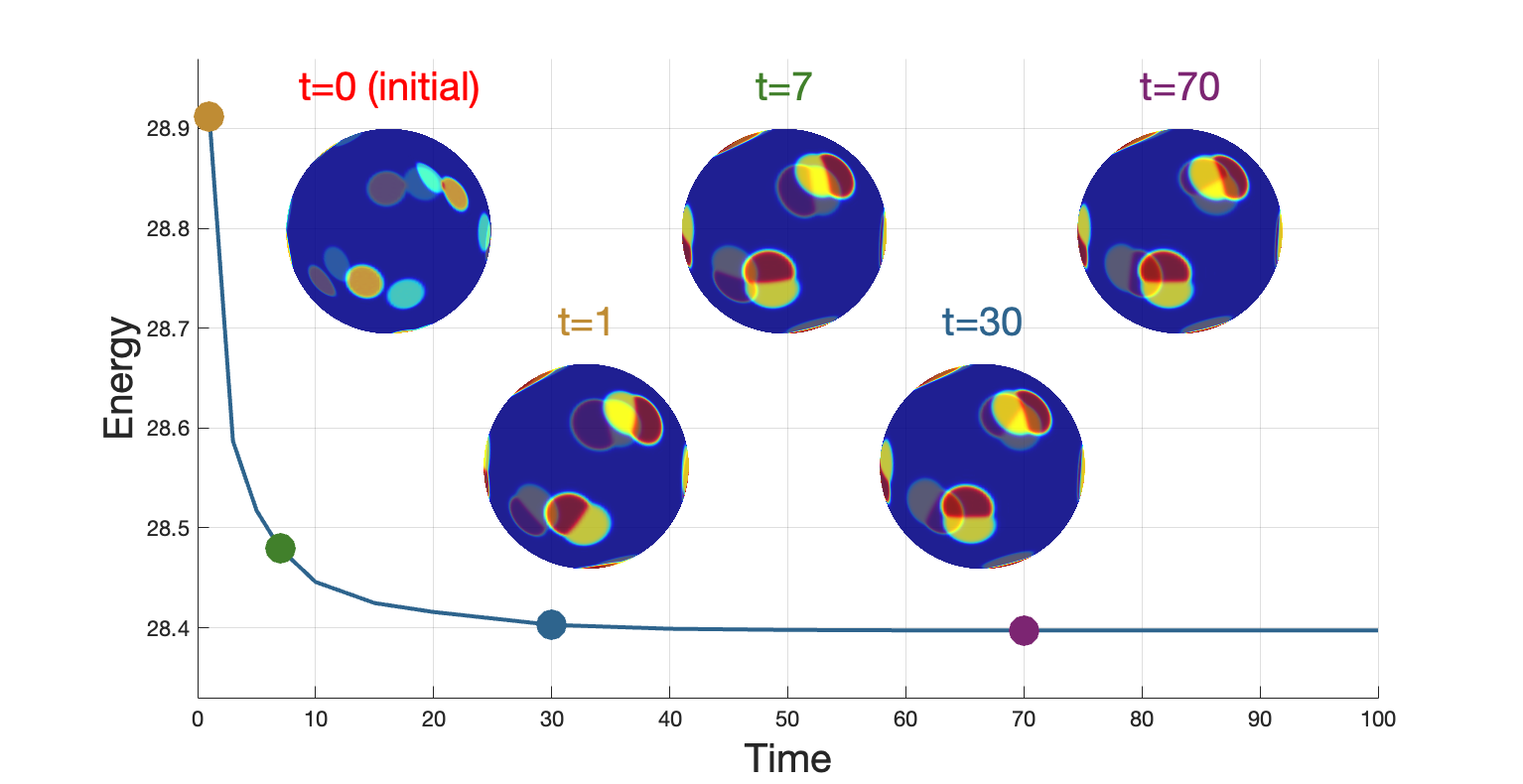}
\includegraphics[width=.8\textwidth]{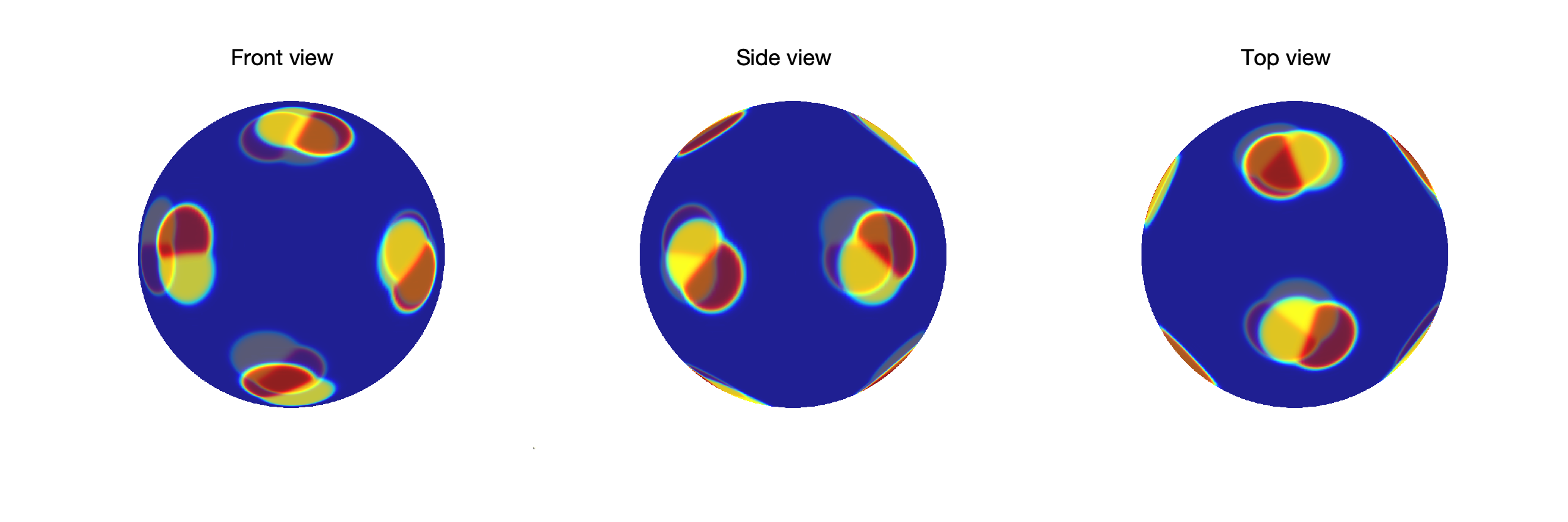}
\end{center}
\caption{(Top) Coarsening dynamics and energy stability for the SNO model with 8 double bubbles, and the corresponding parameters are $\omega_{1} = \omega_{2} = 0.09$, $\gamma_{11} =  \gamma_{22}=500, \gamma_{12} = \gamma_{21}=0$, $\kappa_1 = \kappa_2 = 2000$, $\beta_1 = \beta_2 = 0$ and $M_1 = M_2 = 1000$. (Bottom) Three views (front view, side view, top view) of the equilibrium state with 8 bubbles.}
\label{fig:ternary_system_energy_8b}
\end{figure}

\subsection{Equilibrium States for SOK and SNO Models}\label{sub:Equi_OKNO}

In this subsection, we present the self-assembled patterns at the equilibria for the SOK and SNO models under different repulsive strengths $\gamma$ and $\gamma_{11}$, respectively. 

For the binary system, Figures \ref{fig:binary_system_energy} and \ref{fig:binary_sys_equi} display the self-assembled single-bubble patterns at the equilibria, where the repulsive  force $\gamma$ varies as $\gamma = 1500, 7000, 9000, 13000$ and other parameters are fixed. The equilibria for the SOK model result in assemblies of equally sized and equally separated single bubbles, forming an approximately hexagonal pattern on the unit sphere. As the value of $\gamma$ increases, the stronger repulsive force leads to an increase in the number of bubbles. This trend is depicted in Figure \ref{fig:binary_sys_equi}. From top to bottom, the number of bubbles increases from 48 to 68 to 80 as $\gamma$ increases.
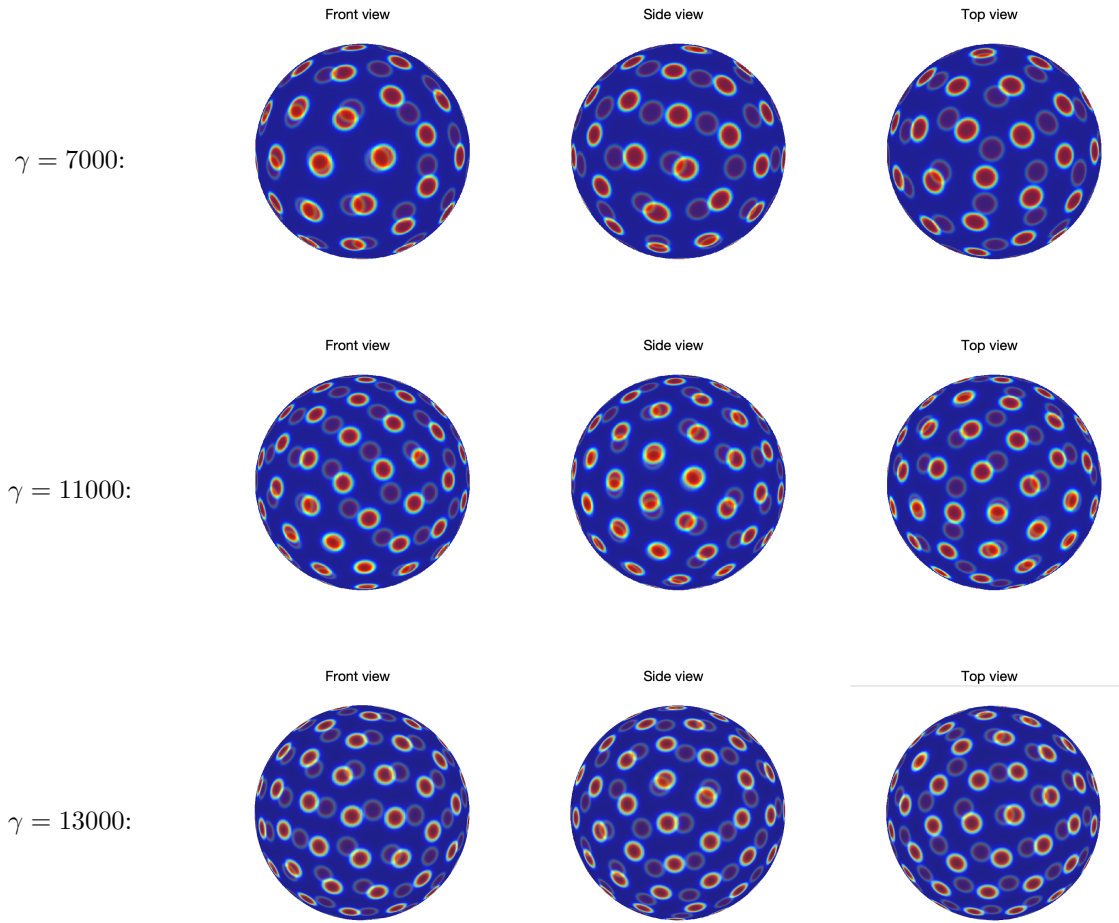
\begin{figure}[htbp]
\begin{center}
\begin{picture}(55,40)
  \put(15, 60){\makebox(0,0){ $\gamma = 7000$:}}
\end{picture}
\includegraphics[width=.8\textwidth]{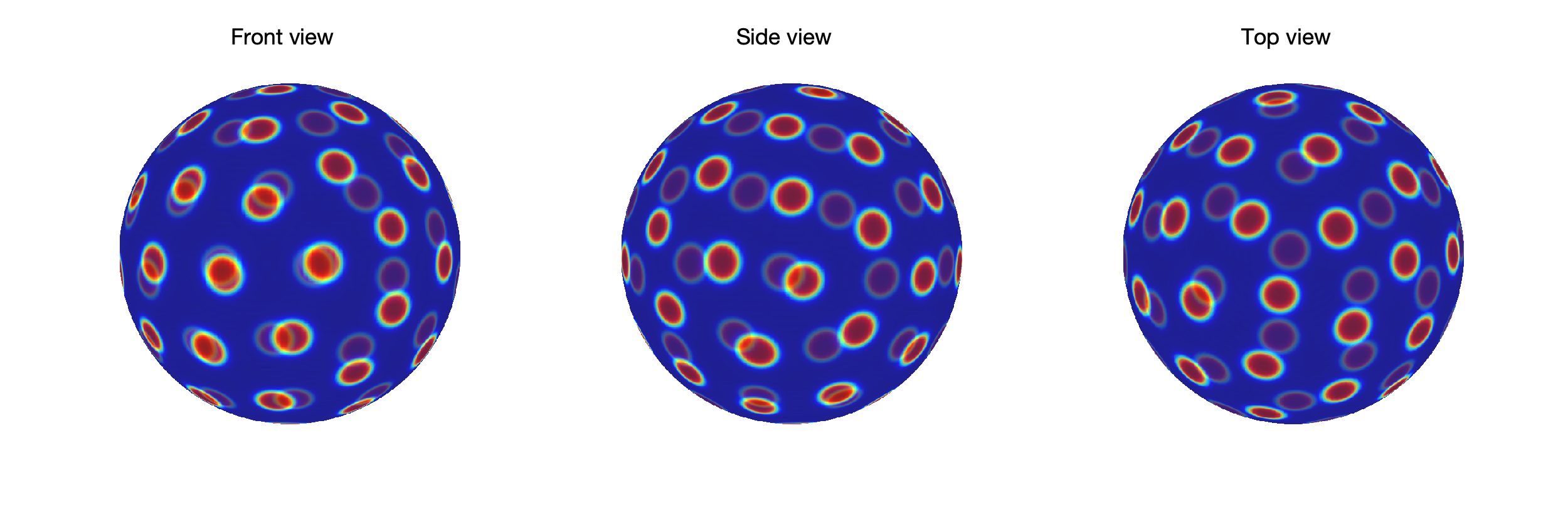}\\
\begin{picture}(55,40)
  \put(15, 60){\makebox(0,0){ $\gamma = 11000$:}}
\end{picture}
\includegraphics[width=.8\textwidth]{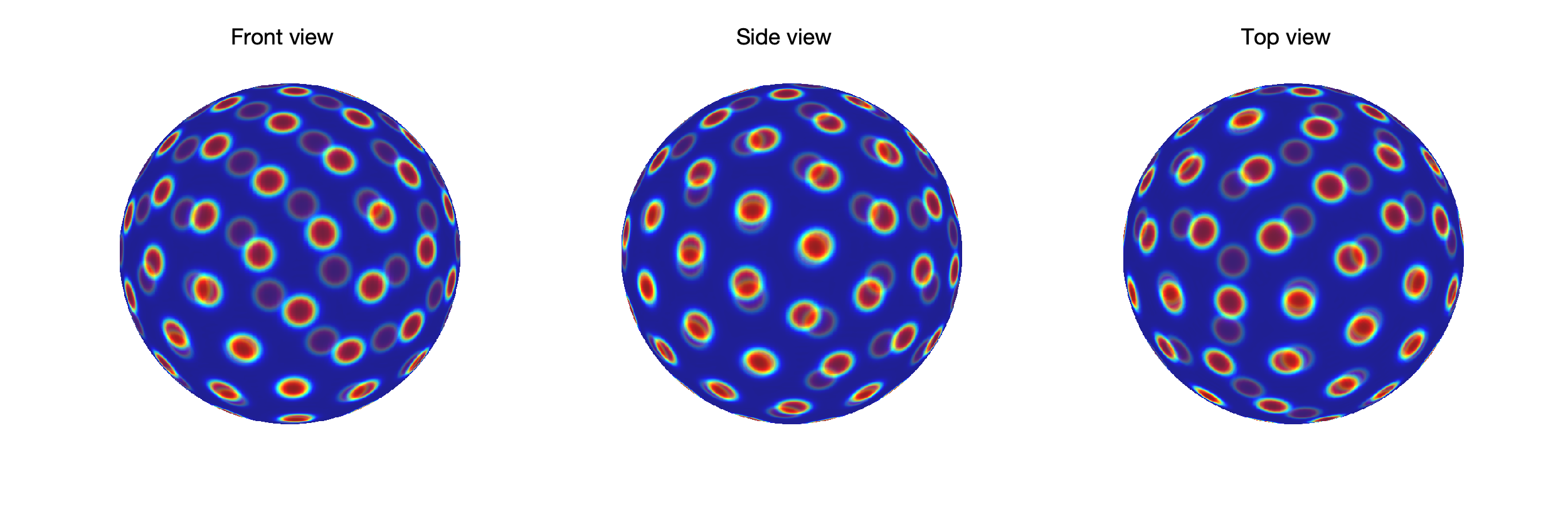}\\
\begin{picture}(55,40)
  \put(15, 60){\makebox(0,0){ $\gamma = 13000$:}}
\end{picture}
\includegraphics[width=.8\textwidth]{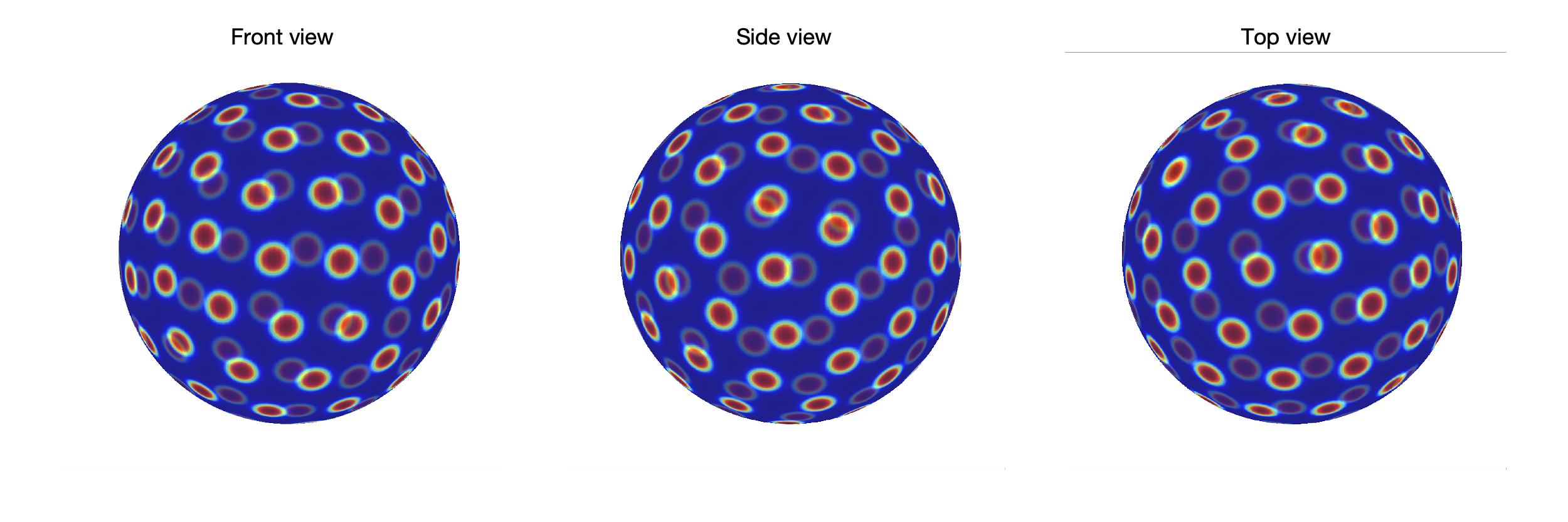}
\end{center}
\caption{Equilibrium states for the SOK model with parameters $\omega = 0.15$, $\kappa = 2000$, $\beta = 0$ and $M = 1000$: (Top) $\gamma=7000$, three views for the same equilibrium state; (Middle) $\gamma=9000$, three views for the same equilibrium state; (Bottom) $\gamma=13000$, three views for the same equilibrium state.}
\label{fig:binary_sys_equi}
\end{figure}


For the ternary system, the equilibrium states are shown in Figures \ref{fig:ternary_system_energy_6b}, \ref{fig:ternary_system_energy_8b}, and \ref{fig:ternary_sys_equi}. Here, the repulsive strengths $\gamma_{11} = \gamma_{22}= 350, 500, 5000, 6000, 7000$, while $\gamma_{12}$ remains fixed at zero, and other parameters are unchanged. The equilibria for the SNO model present the double-bubble patterns with synchronized red and yellow bubbles, and also form an approximately hexagonal structure on the unit sphere. As the value of $\gamma_{11}$ becomes greater, the number of bubbles increases. This effect is displayed in Figure \ref{fig:ternary_sys_equi} with 42, 48, and 52 bubbles from top to bottom for increasing values of $\gamma_{11} = \gamma_{22}$.
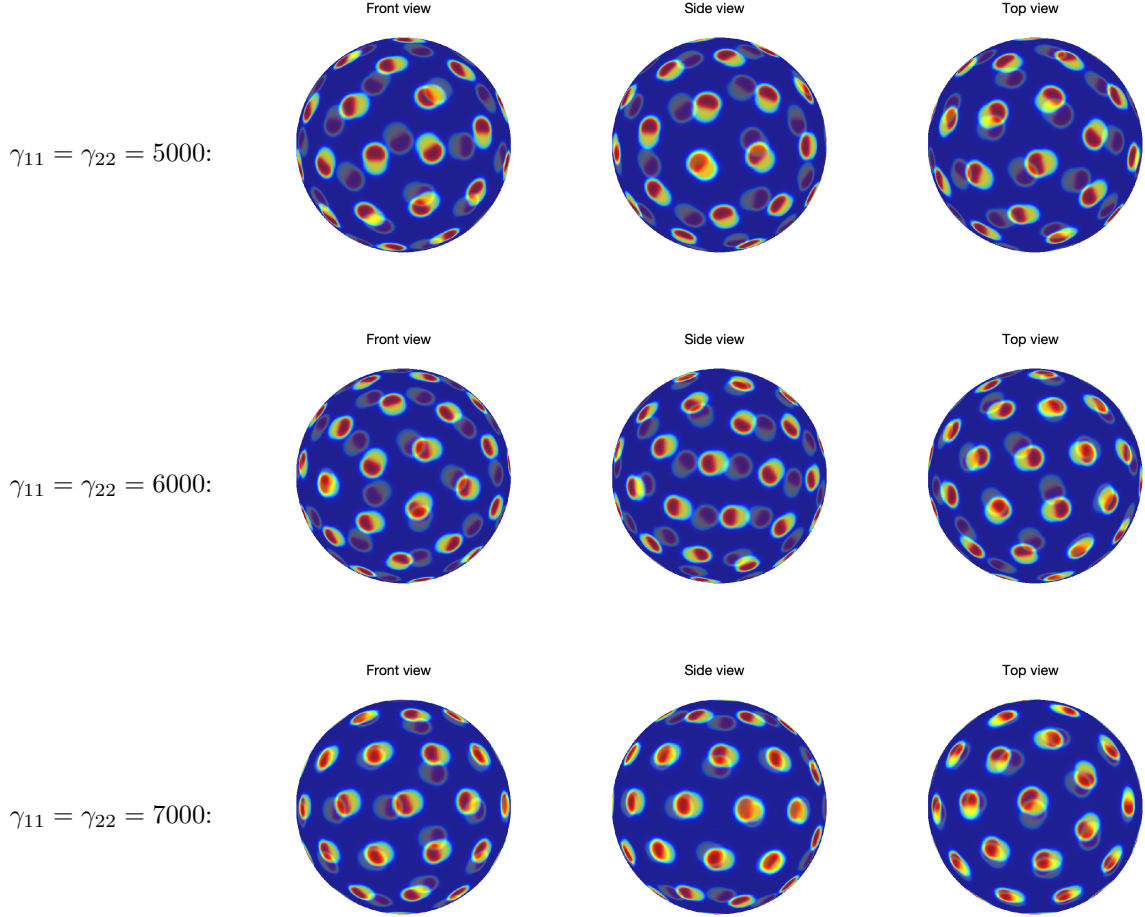
\begin{figure}[htbp]
\begin{center}
\begin{picture}(55,40)
  \put(15, 60){\makebox(0,0){ $\gamma_{11}=\gamma_{22} = 5000$:}}
\end{picture}
\includegraphics[width=.8\textwidth]{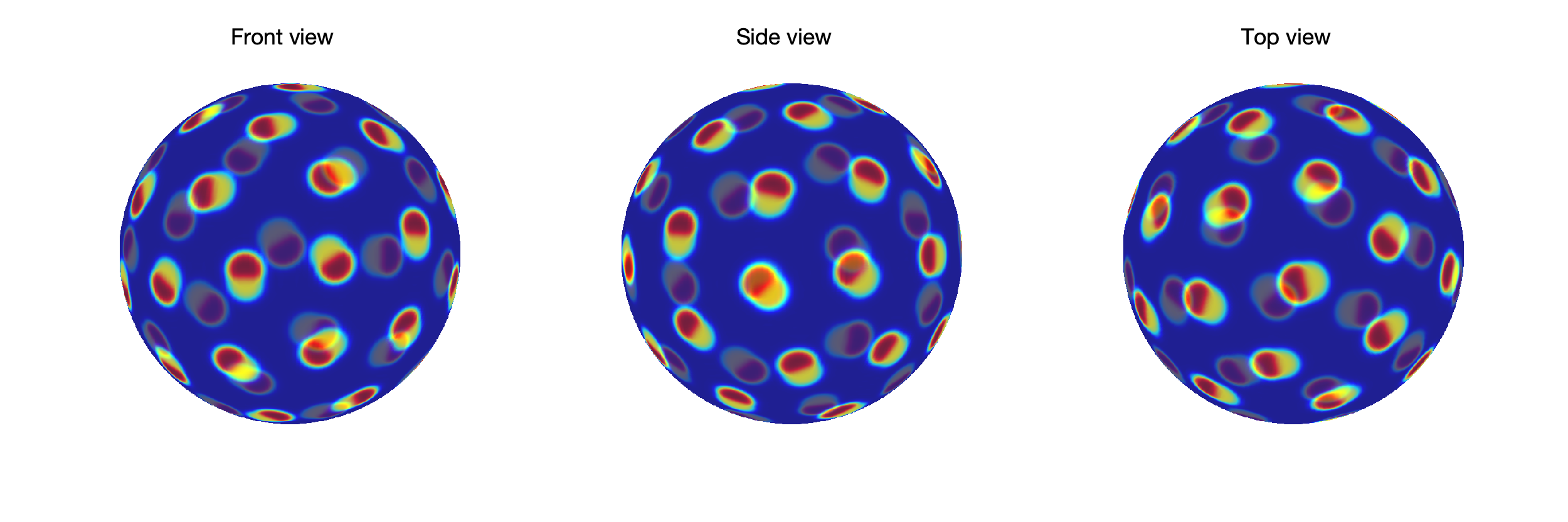}\\
\begin{picture}(55,40)
  \put(15, 60){\makebox(0,0){ $\gamma_{11}=\gamma_{22} = 6000$:}}
\end{picture}
\includegraphics[width=.8\textwidth]{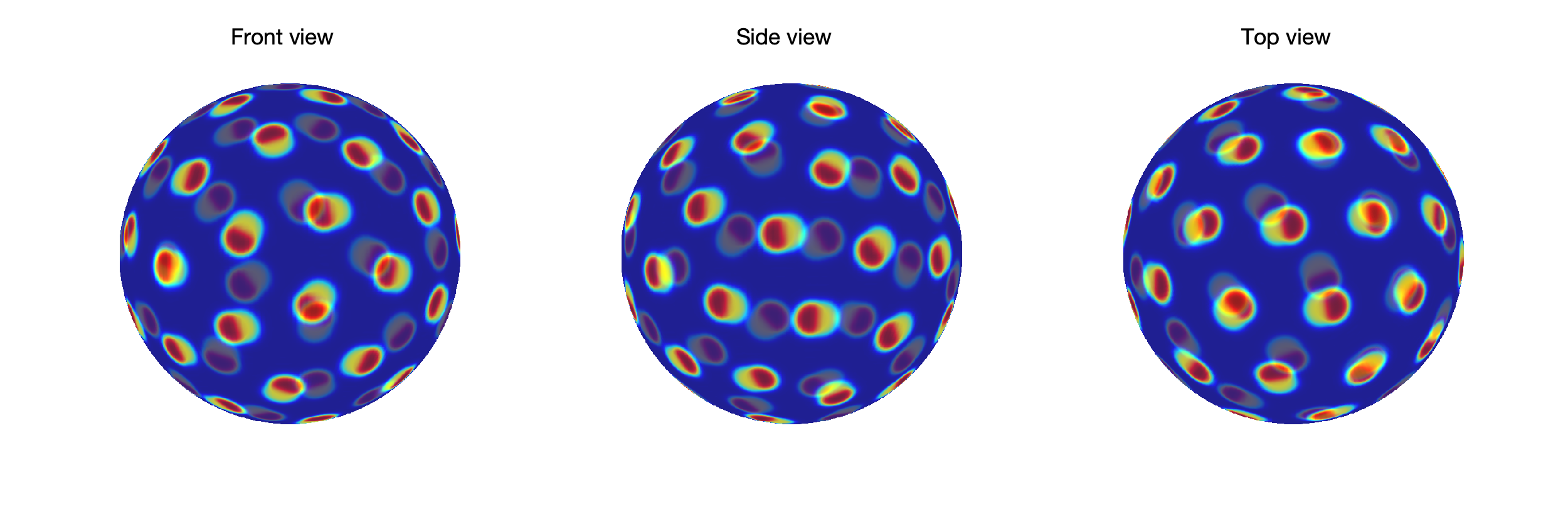}\\
\begin{picture}(55,40)
  \put(15, 60){\makebox(0,0){ $\gamma_{11}=\gamma_{22} = 7000$:}}
\end{picture}
\includegraphics[width=.8\textwidth]{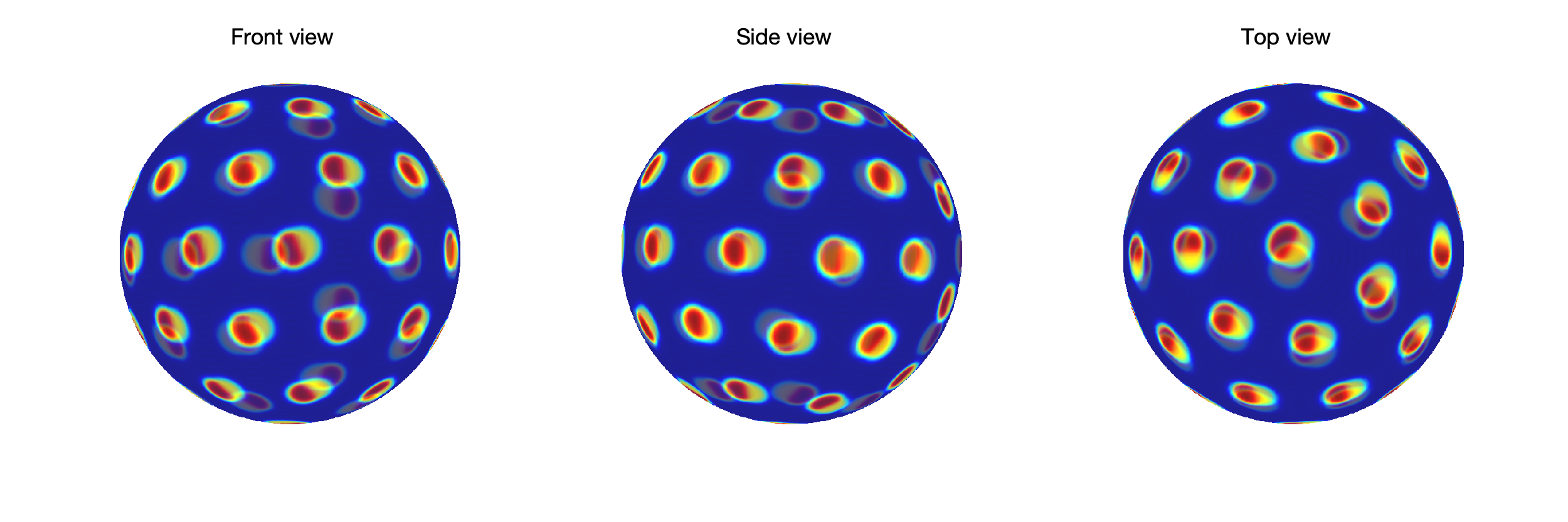}
\end{center}
\caption{Equilibrium states with three views for the SNO model with parameters $\omega_{1} = \omega_{2} = 0.09$, $\kappa_1 = \kappa_2 = 2000$, $\beta_1 = \beta_2 = 0$ and $M_1 = M_2 = 1000$: (Top) $\gamma_{11} = \gamma_{22}=5000$, $\gamma_{12} = \gamma_{21}=0$; (Middle) $\gamma_{11} = \gamma_{22}=6000$, $\gamma_{12} = \gamma_{21}=0$; (Bottom) $\gamma_{11} = \gamma_{22}=7000$, $\gamma_{12} = \gamma_{21}=0$.}
\label{fig:ternary_sys_equi}
\end{figure}

\subsection{The Effect of $\gamma$ for SOK Model and $\gamma_{11} = \gamma_{22}$ for SNO Model}\label{sub:gamma_effect_OKNO}

In this subsection, we investigate the influence of the repulsive strengths $\gamma$ and $\gamma_{11}$ on the number of bubbles in the SOK and SNO models, respectively. Previous work by Wang, Ren, and Zhao \cite{Wang_Ren_Zhao2019} suggested a relationship between the number of bubbles $K_b$ and the repulsive strengths $\gamma_{11}$ for the NO model in the square domain, observing a two-thirds power law, where $K_b \sim \gamma_{11}^{2/3}$. 

Our numerical experiments, as shown in Figure (\ref{fig:Gamma_23}), confirm similar behavior on the spherical surface, where the number of double bubbles $K_b$ in an assembly follows a two-thirds law relationship with the repulsive strengths $\gamma$ and $\gamma_{11}$ in the SOK and SNO models, respectively. The insets on the top figure of Figure (\ref{fig:Gamma_23}) present the equilibrium states for the SOK model with values of $\gamma = 1500, 7000, 13000$, while the bottom insets display the equilibrium states for the SNO model with $\gamma = 500, 4000, 8000$.

Due to the non-convexity of the SOK and SNO models, the steady states of the pACSOK and pACSNO systems often represent local minima, influenced by the repulsive forces and random initial conditions, rather than the global minimum of the energy functionals. To approximate the global minimum more closely, we conduct tens of experiments under the same repulsive force with various random initial conditions and identify the equilibrium state that achieves the lowest energy, which is treated as the global minimum. This approach allows us to estimate the global minimum more accurately and count the corresponding number of bubbles $K_b$, yielding more precise data for the two-thirds law plot as in Figure (\ref{fig:Gamma_23}).

\begin{figure}[htbp]
\begin{center}
\includegraphics[width=.8\textwidth]{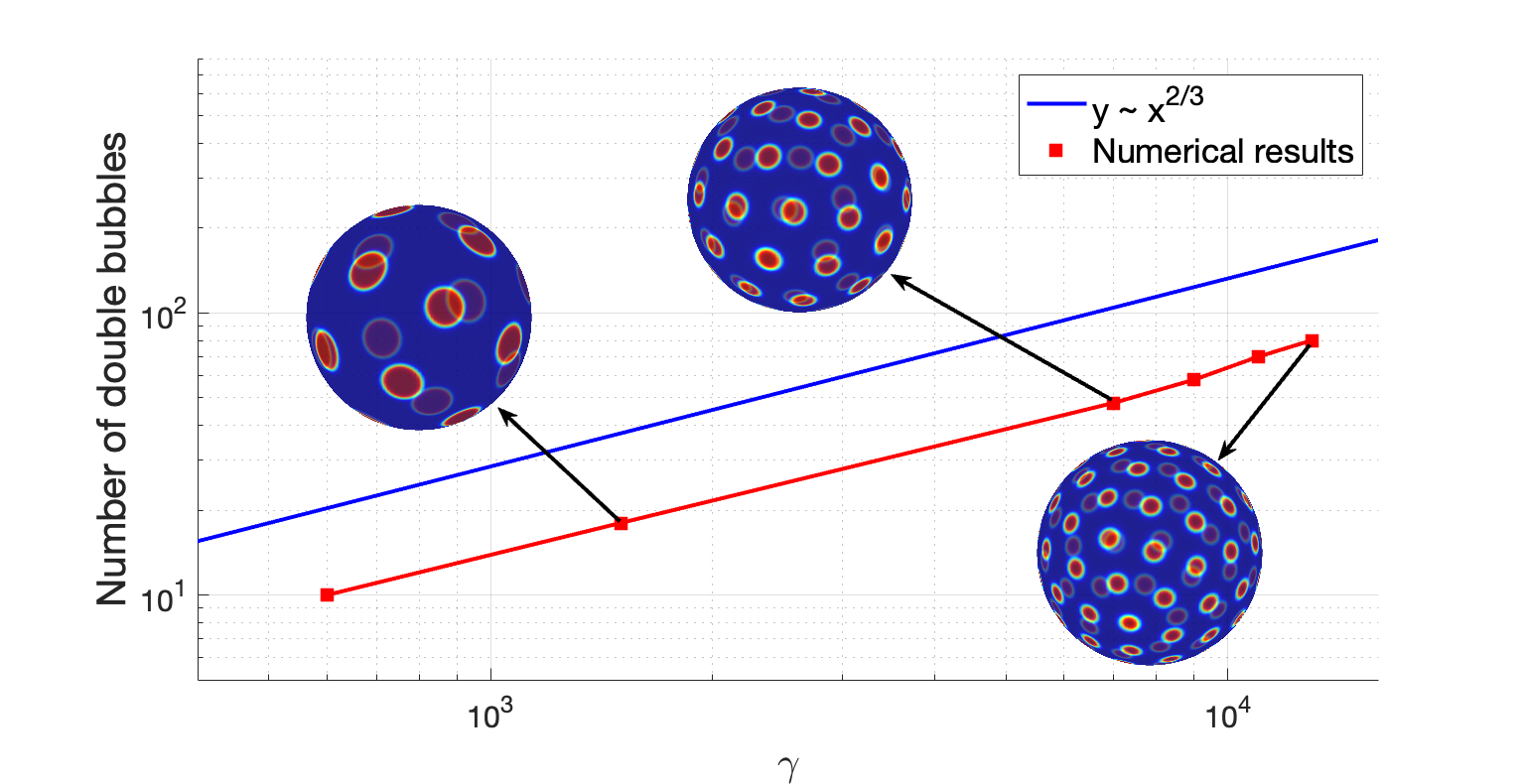}
\includegraphics[width=.8\textwidth]{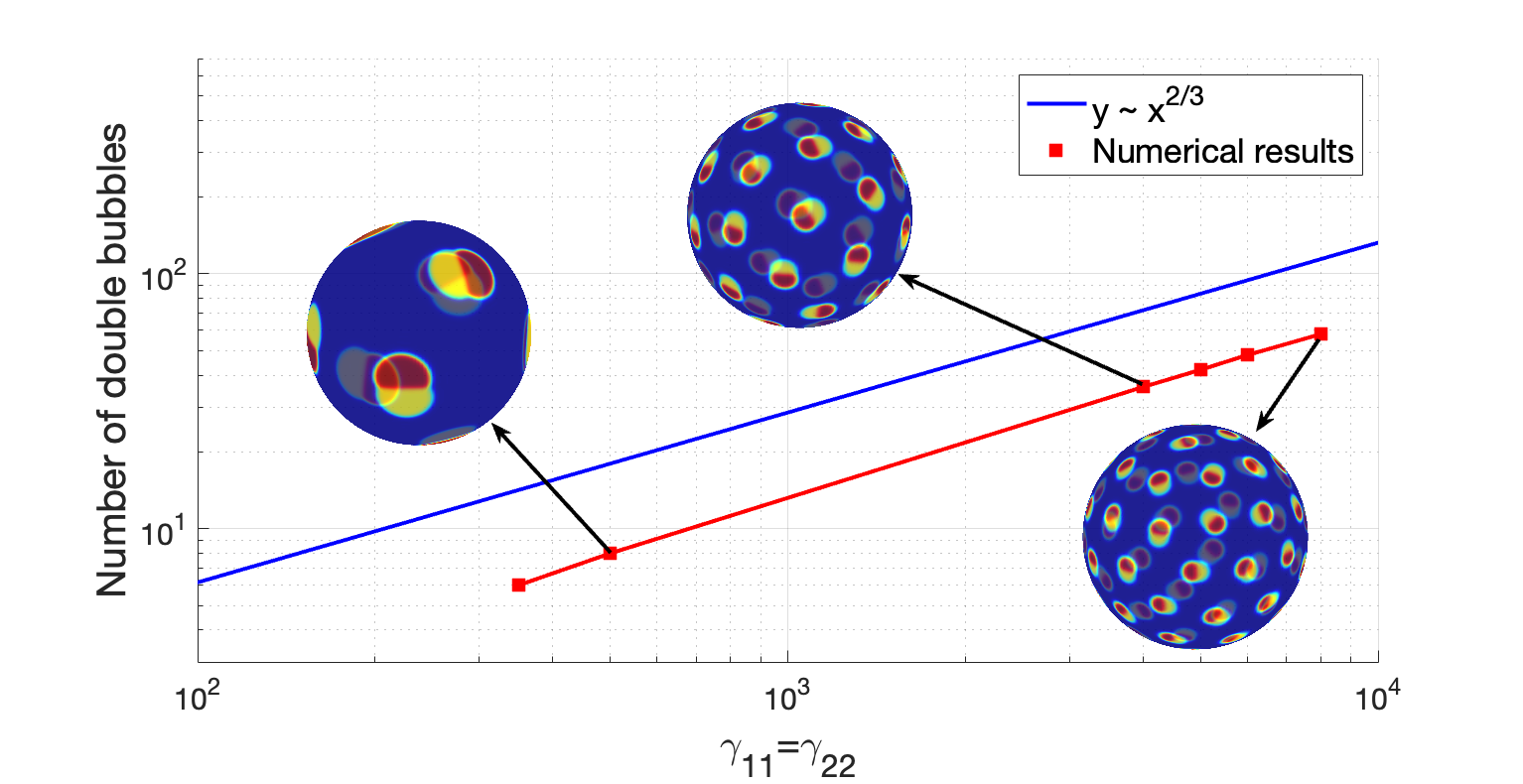}
\end{center}
\caption{(Top) Log-log plot of the dependence of the number of single bubbles for SOK model on $\gamma$. Here $\omega = 0.15$, $\kappa = 2000$, $\beta = 0$ and $M = 1000$. As $\gamma$ increases, the number of bubbles in the assembly grows accordingly. For $\gamma$ = 600, 1500, 7000, 9000, 11000, 13000, the corresponding number of double bubbles are 10, 18, 48, 58, 68 and 80, respectively. (Bottom) Log-log plot of the dependence of the number of double bubbles for SNO model on $\gamma_{11}=\gamma_{22}$. Here $\omega_{1} = \omega_{2} = 0.09$, $\gamma_{12}=\gamma_{21}=0$, $\kappa_1 = \kappa_2 = 2000$, $\beta_1 = \beta_2 = 0$ and $M_1 = M_2 = 1000$. As $\gamma_{11}$ increases, the number of double bubble assemblies grows accordingly. For $\gamma_{11}$ = 350, 500, 4000, 5000, 6000, 8000, the corresponding number of double bubbles are 6, 8, 36, 42, 48 and 58, respectively.}
\label{fig:Gamma_23}
\end{figure}

\subsection{The Effect of $\gamma_{12}$ for SNO Model}\label{sub:gamma12_effect_NO}

In this subsection, we investigate the effect of the repulsive strengths $\gamma_{12}, \gamma_{21}$ on pattern formation in the SNO model and keep $\gamma_{12} = \gamma_{21}$. Figure (\ref{fig:gamma12_effect}) displays the impact of varying $\gamma_{12}$. While the repulsive force $\gamma_{11}$ ($\gamma_{22}$, respectively) promotes the separation of the same component (red and yellow respectively), the repulsive strength $\gamma_{12} = \gamma_{21}$ can be interpreted as the splitting force between red and yellow components, namely, for a fixed value of $\gamma_{11} = \gamma_{22}$, increasing $\gamma_{12}$ tends to separate the red and yellow constituents.

While taking $\gamma_{11}=\gamma_{22} = 6000$, we vary the value of $\gamma_{12}$ as follows: $\gamma_{12} = 0, 2000, 4000, 5000, 6000, 8000$. When $\gamma_{12} = \gamma_{21} = 0$ (Figure \ref{fig:gamma12_effect}, top left), there is no force to separate the red and yellow components; thus, they adhere together, which results in all double-bubble assembly. As $\gamma_{12} = \gamma_{21}$ increases to $2000$ (top middle) and $4000$ (top right), some double bubbles start to separate into single red and yellow bubbles, leading to a coexistence of double-bubble and single-bubble patterns; as the splitting force grows, more single bubbles appear. In Figure \ref{fig:gamma12_effect} (bottom left) with $\gamma_{12} = \gamma_{21}=5000$, the splitting strength is strong enough to fully break all double bubbles, resulting in the purely single-bubble assembly. Further increasing $\gamma_{12} = \gamma_{21}$ to $6000$ forces some red and yellow single bubbles to form distinct clusters. At last, with $\gamma_{12} = \gamma_{21} = 8000$ (bottom right), the repulsive force is strong enough to entirely separate the red and yellow bubbles into distinct regions. Although there are no theoretical studies of the SNO model on the spherical domain, our numerical results are consistent with the findings in square and disk domains \cite{Wang_Ren_Zhao2019,LuoZhao_AAMM2024}.

\begin{figure}[htbp]
\begin{center}
\includegraphics[width=1\textwidth]{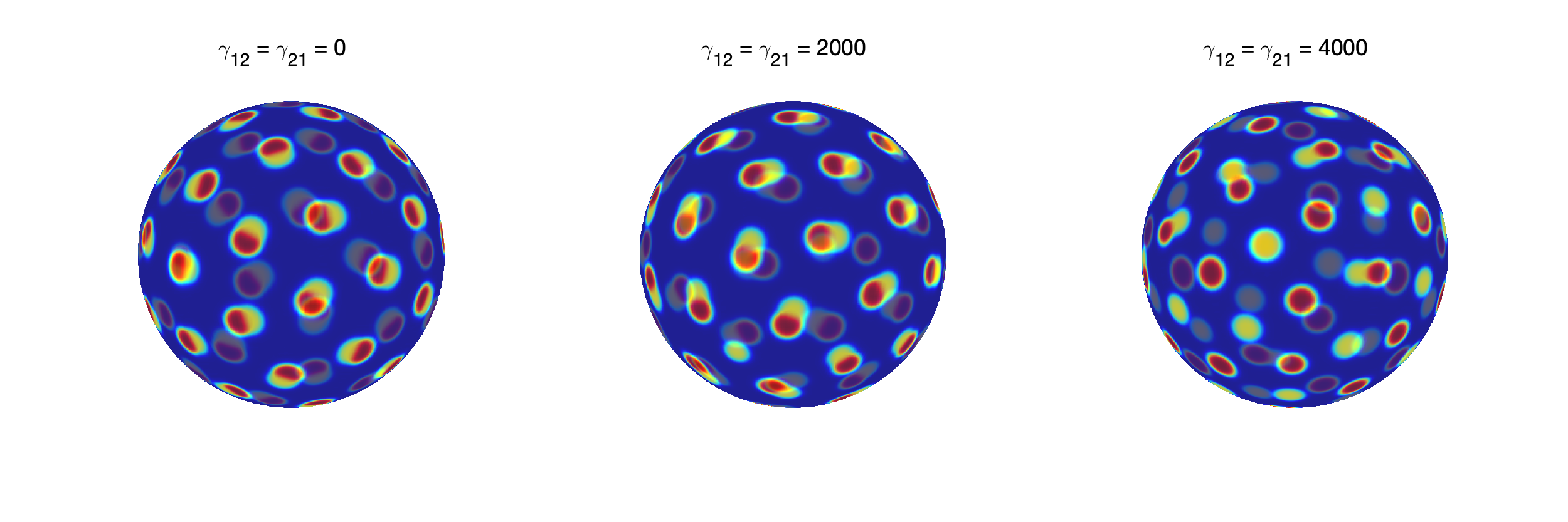}
\includegraphics[width=1\textwidth]{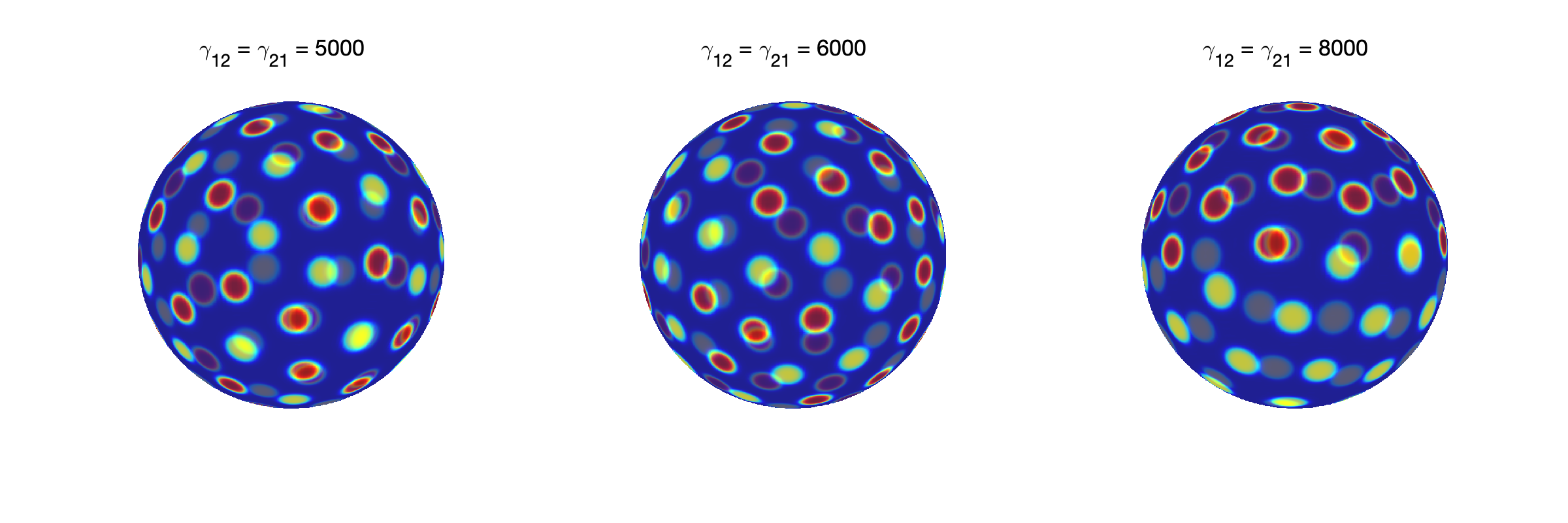}
\end{center}
\caption{The influence of $\gamma_{12} = \gamma_{21}$ to the pattern formation for the SNO model with parameters $\omega_{1} = \omega_{2} = 0.09$, $\gamma_{11} = \gamma_{22} = 6000$, $\kappa_1 = \kappa_2 = 2000$, $\beta_1 = \beta_2 = 0$ and $M_1 = M_2 = 1000$. The value of $\gamma_{12}$ for each case is presented on the top of each subplot.}
\label{fig:gamma12_effect}
\end{figure}


\section{Concluding Remarks}

In this paper, we develop a numerical method for the pACSOK and pACSNO systems to investigate coarsening dynamics and pattern formations of the SOK and SNO models. We apply the DFS method for spatial discretization and use the BDF2 scheme for time evolution to construct an energy-stable scheme for the pACSOK and pACSNO systems. 

In our numerical experiments, we observe phase separation and single-bubble assemblies on the spherical surface using pACSOK dynamics for the binary system when one species occupies a much smaller volume than the other. We notice the hexagonal structure in the self-assembled patterns for the SOK model, which successfully resembles the experimental patterns for multicomponent biomembranes \cite{Baumgart_Nature2003}. For the pACSNO dynamics, we observe that varying repulsive forces between species lead to different bubble assemblies, including all-double-bubble, all-single-bubble, and coexistence of single-bubble and double-bubble patterns. Meanwhile, we numerically investigate the relation between the number of bubbles and the repulsive strengths, which follows an increasing two-thirds law \cite{Wang_Ren_Zhao2019}.

In the future work, we aim to develop a model that couples the binary system of the OK model on a deformable surface. This model combines the energy functional of the OK model (\ref{eqn:OK}) with the phase field elastic bending energy  \cite{WangDu_MathBio2008}
\begin{align*}
E^{\text{bending}} = \int_{\Omega}\kappa(u)H(u)^2 \ dx,
\end{align*}
where the function $u$ represents the density of one component of the biomembranes, and $H(u)$ denotes the mean curvature in phase field formulation, and both volume and surface area of the biomembranes are fixed. Specifically, we will investigate various patterns of the OK model on deformable surfaces for various values of system parameters. Such models may reproduce additional experimental observations in multicomponent biomembranes \cite{Baumgart_Nature2003}. Moreover, we will continue studying pattern formations of the NO model theoretically, investigating how the parameters influence phase separation in the ternary system on disks and spheres. In addition, while the current numerical schemes for the OK and NO models are limited to second-order accuracy in time, we would like to develop higher-order algorithms, such as fourth-order backward differentiation formulas (BDF4) or exponential time differencing schemes (ETDRK4) to improve the accuracy of numerical simulations. Furthermore, it is also an interesting future work to develop the maximal principle preserving (MPP) schemes for the ACOK and ACNO equations on different domains.


\section*{Acknowledgments} 

This work is supported by a grant from the Simons Foundation through Grant No. 357963 and NSF grant DMS-2142500.

\section*{Data Availability} 

Datasets generated in the current study are available from the corresponding author on request. 

\section*{Declarations} 

\textbf{Conflict of Interest:} The authors declare that they have no conflict of interest.


\bibliography{citation}

\end{document}